\documentclass[12pt]{amsart}
\usepackage{amsmath}
\usepackage{amssymb,amsmath,amscd}
\usepackage{bm}
\usepackage{amsfonts}
\usepackage{epsf}
\usepackage{graphicx}
\input epsf%
\usepackage{multirow}

\usepackage{geometry}
\usepackage{stmaryrd}

\usepackage{soul} 
\usepackage{color, xcolor}

\newtheorem{theorem}{Theorem}[section]
\newtheorem{lemma}[theorem]{Lemma}

\newtheorem{remark}[theorem]{Remark}

\newcommand{\eps}{\varepsilon}

\newcommand{\T}{\mathcal{T}}
\newcommand{\M}{\mathcal{M}}
\newcommand{\R}{\mathbb{R}}

\newcommand{\Sym}{\mathbb{S}}
\newcommand{\dv}{\operatorname{div}}
\newcommand{\Span}{\text{span}}

\newcommand{\norm}[1]{\left\lVert#1\right\rVert}

\newcommand{\jumpE}[2]{\left\llbracket#1\right\rrbracket_{#2}}
\newcommand{\dx}[1]{\mathrm{d}#1}
\newcommand{\abs}[1]{\left|#1\right|}
\newcommand{\ip}[2]{\left(#1\right)_{#2}}

\newcommand{\trace}[2]{\left.#1\right|_{#2}}

\numberwithin{equation}{section}

\long\def\comment#1{}

\usepackage{tikz}
\usetikzlibrary{calc}
\usetikzlibrary{angles,quotes}

\usepackage{subcaption}
\usepackage{overpic}                
\usepackage{tikz-3dplot}

\definecolor{color1}{RGB}{52, 152, 219}   
\definecolor{color2}{RGB}{230, 126, 34}   
\definecolor{color3}{RGB}{155, 89, 182}   
\definecolor{color4}{RGB}{46, 204, 113}   
\definecolor{color5}{RGB}{231, 76, 60}    
\definecolor{color6}{RGB}{241, 196, 15}   

\begin{document}
	
	\title[]{
		Lower order mixed elements for the linear elasticity problem in 2D and 3D
	}
	
	
	\author {Jun Hu}
	\address{LMAM and School of Mathematical Sciences, Peking University,
		Beijing 100871, P. R. China.\\ Chongqing Research Institute of Big Data, Peking University, Chongqing 401332, P. R. China. hujun@math.pku.edu.cn}
	
	\author{Rui Ma}
	\address{Beijing Institute of Technology,
		Beijing 100081, P. R. China. rui.ma@bit.edu.cn}
	
	\author{Yuanxun Sun}
	\address{LMAM and School of Mathematical Sciences, Peking University,
		Beijing 100871, P. R. China. 1901112048@pku.edu.cn}

	\thanks{The first author was supported by NSFC
		project 12288101, the second author was supported by NSFC project 12301466.}
	
	\begin{abstract}
		
		\vskip 15pt
		
		In this paper, we construct two lower order mixed elements for the linear elasticity problem in the Hellinger-Reissner formulation, one for the 2D problem and one for the 3D problem, both on macro-element meshes. The discrete stress spaces enrich the analogous $P_k$ stress spaces in {[J. Hu and S. Zhang, arxiv, 2014, J. Hu and S. Zhang, Sci. China Math., 2015]} with simple macro-element bubble functions, and the discrete displacement spaces are discontinuous piecewise $P_{k-1}$ polynomial spaces, with $k=2,3$, respectively. Discrete stability and optimal convergence is proved by using the macro-element technique. As a byproduct, the discrete stability and optimal convergence of the $P_2-P_1$ {mixed} element in [L. Chen and X. Huang, SIAM J. Numer. Anal., 2022] in 3D is proved on another macro-element mesh. For the mixed element in 2D, an $H^2$-conforming composite element is constructed and an exact discrete elasticity sequence is presented. {Numerical experiments confirm the theoretical results.}

		\vskip 15pt
		
		\noindent{\bf Keywords.}{
			linear elasticity, lower order mixed elements, macro-element, $H^2$-conforming composite element, discrete elasticity sequence.}
		
		\vskip 15pt
		
		\noindent{\bf AMS subject classifications.}
		{ 65N30, 74B05.}
	\end{abstract}
	\maketitle
	
	\section{Introduction}

	For a bounded polygonal domain $\Omega\subset\R^n$ with $n=2$ {(resp. a bounded polyhedral domain for $n=3$)}, the linear elasticity {problem} based on the Hellinger-Reissner variational principle reads: Given $f\in V$, find $(\sigma,u)\in \Sigma\times V$  such that
	{\begin{equation}
		\left\{
		\begin{aligned}
		(A\sigma,\tau)_{\Omega}+(\dv\tau,u)_{\Omega}&=0,\\
		(\dv\sigma,v)_{\Omega}&=(f,v)_{\Omega}
		\end{aligned}
		\right.
		\label{continuousP}
		\end{equation}}holds for any $(\tau,v)\in\Sigma\times V$. Here $\Sigma:=H({\rm div},\Omega;\mathbb {S})$ is {the space of symmetric stress fields}, $V:=L^2(\Omega;\mathbb{R}^n)$ is {the space of displacement fields}, the compliance tensor $A(x):\Sym\rightarrow\Sym$ is bounded and symmetric positive
	definite uniformly for $x\in\Omega$, with $\Sym:=\R^{n\times n}_{\text{sym}}$ being the set of symmetric tensors. The space $H(\dv,\Omega;\Sym)$ is defined by
	\begin{equation*}
	H(\dv,\Omega;\Sym):=\{\tau\in L^2(\Omega;\Sym):\ \dv\tau\in L^2(\Omega;\R^n)\},
	\end{equation*}
	equipped with the norm
	\begin{equation*}
		\norm{\tau}_{{H(\dv)}}^2:=\norm{\tau}_{0}^2+
		\norm{\dv\tau}_{0}^2.
		\end{equation*}
		
	It is a challenge to design stable discretizations for problem \eqref{continuousP}, due to the additional symmetry constraint on the stress tensor {$\sigma$}. Earlier works focused on designing composite elements, in which the stress is discretized on a more refined triangulation than the displacement, see \cite{Arnold1984, Johnson-Mercier}. In \cite{Arnold-Winther-conforming}, Arnold and Winther {proposed} the first family of mixed finite elements based on polynomial shape function spaces. From then on, various stable mixed elements with strong or weak symmetry have been constructed for both the 2D problem and the 3D problem, see \cite{Adams,Arnold-Awanou,Arnold-Awanou-Winther,Arnold-Falk-Winther,Boffi-Brezzi-Fortin,CGG}.

	In \cite{HuZhang2014a,HuZhang2015tre}, Hu and Zhang constructed the stable pair for triangular meshes {in 2D} and tetrahedral meshes {in 3D}. The lowest polynomial {degrees} of the mixed elements for the 2D and 3D problems are $k=3$ and $k=4$, respectively. The generalization to any dimension can be found in \cite{Hu2015trianghigh}, and the lowest polynomial  {degree} is $k=n+1$. Mixed elements with reduced number of degrees of freedom (DoFs) were constructed in \cite{HuZhang2015trianglow}. The stress spaces therein enrich the $P_k$ ($2\leqslant k\leqslant n$) analogous elements in \cite{Hu2015trianghigh} {with} face bubble functions of degree $n+1$.
	
	Macro-element techniques are widely used for designing {mixed} finite elements for the linear elasticity problem. For the 2D problem, the {first order} composite element proposed in \cite{Johnson-Mercier} and the {higher order} composite elements in \cite{Arnold1984} use the Hsieh-Clough-Tocher (HCT) meshes for the discretization for stresses, where each triangle consists of three subtriangles. For the 3D problem, lower order mixed elements merely using $P_k$ ($k<4$) polynomials for the stress were recently constructed on macro-element meshes, such as $k\geqslant 2$ for {HCT} splits \cite{Alfeld} and $k\geqslant 1$ for Worsey-Farin splits \cite{Worsey}, and each macro-element therein consists of four and twelve elements, respectively. In \cite{Gong}, mixed elements {using $P_k$ ($k=n-1,n$) polynomials for the stress were constructed on the HCT meshes} for both the 2D problem and the 3D problem. {Recently in \cite{3DP3}, {a} mixed element merely using $P_3$ polynomials for the stress was constructed on tetrahedral meshes, under some mild mesh conditions.}

	{This paper proposes a new} mixed element of  {degree} $k=2$ for the 2D problem and {a} mixed {element} of degree $k=3$ for the 3D problem to complete {the} family of elements {in \cite{Hu2015trianghigh}} on some  {simple} macro-element meshes,  {where} each macro-element consists of four elements for both cases, and the meshes can be obtained by {uniformly bisecting an initial simplicial mesh twice.}  {The} construction of the new mixed elements is  {motivated by the idea of} \cite{HuMa} {that it} is possible to partially relax the $C^0$ continuity of discrete stresses at some vertices {owing to} the special mesh structures. The new mixed methods can be easily implemented, since they have the same local stiffness matrices as the {lower order analogy of the} mixed methods in {\cite{HuZhang2014a,HuZhang2015tre}}. One only needs to rearrange part of the DoFs. The discrete inf-sup condition is proved by {applying} the two-step method in \cite{Hu2015trianghigh,HuZhang2014a,HuZhang2015tre} to the {macro-elements}. The two-step method consists of the construction of a stable interpolation with partial commuting diagram property \cite[Lemma~3.1]{Hu2015trianghigh}, and the characterization of the divergence of the $H(\dv;\Sym)$-bubble functions \cite[Theorem~2.2]{Hu2015trianghigh}. In this paper, these results are applied to {some} macro-element meshes: the first step utilizes the geometric structure of the macro-elements to construct the interpolation operator {to deal with} lower degree polynomials; the second step uses some new macro-element bubble functions to deal with the orthogonal complement space of the rigid motion space. Similar techniques have been used in \cite{Rectangle} to prove the discrete stability of rectangular elements for the pure traction boundary problem. With this, the discrete inf-sup condition can be established as well as the discrete stability and optimal convergence.

	The first step can be easily verified, and the difficulty lies in the second step: the macro-element bubble function spaces need to enrich the piecewise $H(\dv;\Sym)$-bubble spaces so that the divergence operator maps them onto the orthogonal complement space of the rigid motion space. This is quite technical for complicated macro-elements, {especially} for the {mixed} element of degree $k=2$ in 3D. Therefore, another approach is presented to {avoid tedious analysis}.  {With} the matrix Piola transform {\cite{Arnold-Winther-conforming,Johnson-Mercier}}, the analysis  {is carried out} on a reference macro-element  {by} {computing} the rank of certain matrices. This can be carried out directly on computers. Following this computer-assisted proof, the second step of the two-step method can be established, and the  {$P_2-P_1$ mixed element} of \cite{ChenHuang2022} is proved to be stable, on another macro-element mesh,  {where} each macro-element consists of twelve tetrahedral elements.


	For $\Omega\subset\R^2$, the airy function $J:\ H^2(\Omega)\rightarrow H(\dv,\Omega;\Sym)$ is defined as
	\begin{equation*}
	Jq:=\begin{bmatrix}
	\partial_{yy}^2q&-\partial_{xy}^2q\\
	-\partial_{xy}^2q&\partial_{xx}^2q
	\end{bmatrix},
	\end{equation*}
	{which plays an essential role in the elasticity complex in 2D.} Several two dimensional $H\left(\dv,\Omega;\Sym\right)$ conforming finite element spaces can be placed in exact finite element {elasticity} sequences starting from the related conforming $H^2$ elements. For example, the composite element of \cite{Johnson-Mercier} is related to the Clough-Tocher composite $H^2$ element, the composite elements in \cite{Arnold1984} are related to the higher order composite $H^2$ elements, the Arnold-Winther elements and the elements in \cite{HuZhang2014a} are related to the Argyris element \cite{Arnold-Winther-conforming, Nodal}. {This paper presents a new} $H^2$ conforming finite element related to the newly proposed lower order element on the {aforementioned} macro-element mesh {and constructs an exact discrete sequence}. {Interested readers can find some} finite element elasticity complexes in three dimensions, on tetrahedral meshes \cite{ChenHuang3D2022}, cuboid meshes \cite{HuLiangLin}, the HCT splits \cite{Alfeld}, and the Worsey-Farrin splits \cite{Worsey}.

	The rest of the paper is organized as follows. In Section 2, some notation and results are prepared. In Section 3, the mixed elements of  {degree} $k=2$ for the 2D problem and  {degree} $k=3$ for the 3D problem are {presented} on some macro-element meshes. {The discrete stability} and optimal convergence of the $P_2-P_1$ {mixed} element of \cite{ChenHuang2022}  {is} proved on macro-element meshes for the 3D problem. In Section 4, an $H^2$ conforming finite element space related to the 2D {mixed element} is constructed and an exact sequence is presented. In Section 5, numerical results are {provided} to confirm the theoretical analysis.


	Throughout this paper, let $H^k(\omega;{W})$ denote the Sobolev space consisting of functions with domain  {$\omega\subset\R^2$ or $\omega\subset\R^3$}, taking values in the finite-dimensional vector space ${W}$ ($\R$,  {$\R^2$ or $\R^3$,} or $\Sym$), and with all derivatives of order at most $k$ square-integrable. Let $H_0(\dv,\omega;\Sym)$ denote the space of functions in $H(\dv,\omega;\Sym)$ that has vanishing trace, \emph{i.e.}
		\begin{equation*}
		H_0(\dv,\omega;\Sym):=\{\tau\in H(\dv,\omega;\Sym):\ \trace{\tau\bm{n}}{\partial \omega}=0\}.
		\end{equation*}
		Here $\bm{n}$ denotes the outward unit normal vector of $\omega$. Let $\norm{\cdot}_{k,\omega}$ and $\abs{\cdot}_{k,\omega}$ be the norm and the seminorm of $H^k(\omega;{W})$, and let $\left(\cdot,\cdot\right)_{\omega}$ denote the inner product of $L^2(\omega)$. When $\omega=\Omega$, the norm and the seminorm are simply denoted by $\norm{\cdot}_k$ and $\abs{\cdot}_k$. Let $P_k(\omega;{W})$ denote the space of polynomials of degree at most $k$, taking values in the space ${W}$.  {When there is no ambiguity, an element is called $P_k$ if its shape function space contains polynomials of degree at most $k$.}

	\section{{Preliminaries and notation}}

	This section gives the general definition of a class of mixed finite elements on the macro-element meshes for the linear elasticity problem \eqref{continuousP}. Some notation and results are prepared for later use.


	\subsection{{Some notation}}
	\label{subsec: notation}

	Assume that the domain $\Omega$ {can} be {divided} into macro-elements, with each macro-element being the union of a fixed number of simplicial elements. Let $\M_h$ denote the collection of all the macro-elements. For an edge $E$ {in 2D} (resp. a face $F$ {in 3D}), let $K_+$ and $K_-$ be the elements sharing $E$ (resp. $F$). For a piecewise smooth function $v$, define the jump of $v$ across $E$ (resp. $F$) to be {$\jumpE{v}{E(\text{resp. $F$})}:=v_+-v_-$}. Note that the jump defined here depends on the choice of $K_+$. For any vertex {$X$}, let {$\lambda_{X}$} be the {continuous} piecewise linear  {nodal} basis function associated with {$X$}. For any given integer $k$ and the Lagrange node $X$, let $\Phi_X$ be the $P_k$ Lagrange basis function associated with $X$. For a given macro-element $M\in\M_h$, {the space of rigid motions reads:}
	 {\begin{equation*}
		RM\left(M\right):=\left\{v\in H^1\left(M\right):\ \eps\left(v\right):=\frac{1}{2}(\nabla v+(\nabla v)^T)=0\right\}.
		\end{equation*}}
		Rigid motion functions are linear functions taking the following form:
		\begin{equation}
		v=\begin{cases}
		\begin{bmatrix}
		a_1\\
		a_2\\
		\end{bmatrix}+
		\begin{bmatrix}
		0&b_{12}\\
		-b_{12}&0\\
		\end{bmatrix}
		\begin{bmatrix}
		x\\
		y\\
		\end{bmatrix},\quad\text{$n=2,$}\\ \\
		\begin{bmatrix}
		a_1\\
		a_2\\
		a_3
		\end{bmatrix}+
		\begin{bmatrix}
		0&b_{12}&b_{13}\\
		-b_{12}&0&b_{23}\\
		-b_{13}&-b_{23}&0
		\end{bmatrix}
		\begin{bmatrix}
		x\\
		y\\
		z
		\end{bmatrix},\quad\text{$n=3$}.
		\end{cases}
		\label{eq: RM}
		\end{equation}

		Define the following piecewise $P_{k-1}$ function space
			\begin{equation*}
			V_M:=\{v\in L^2\left(M;\R^n\right):\ \trace{v}{K}\in P_{k-1}\left(K;\R^n\right),\ \forall K\subset M\}.
			\end{equation*}
			Define the {$L^2$} orthogonal complement space of $RM(M)$ with respect to $V_M$ {as}
			\begin{equation*}
			\begin{aligned}
			RM^{\perp}(M)&=\{v\in V_M:\ (v,w)_{{M}}=0,\ \forall w\in RM(M)\}.
			\end{aligned}
			\end{equation*}
		

		\subsection{Piola transform}
		\label{subsec: Piola}
		
		{This subsection introduces the Piola transform for the {$H(\dv;\Sym)$-tensors} and the  rigid motion functions. These transforms will be used in the scaling argument and the proof of the discrete inf-sup condition for the $P_2-P_1$ element in Section \ref{subsec: 3DP2}.

			Let $\widehat{K}$ be a reference element. Let $\mathcal{F}:\ \widehat{K}\rightarrow K$ be an affine isomorphism of the form $\mathcal{F}\widehat{x}=B\widehat{x}+b$. The gradient operator, the divergence operator and the {symmetric gradient} operator on the reference element are denoted by $\widehat{\nabla}$, $\widehat{\dv}$ and $\widehat{\eps}$, respectively.

			Given a differentiable function $\widehat{v}:\ \widehat{K}\rightarrow\R^n$, the covariant transform in \cite[Eq.~(2.1.82)]{Boffi-Brezzi-Fortin2013} defines $v:\ K\rightarrow\R^n$  {as follows}
			\begin{equation}
			v(x):=B^{-T}\widehat{v}\left(\widehat{x}\right).
			\label{transform-vector}
			\end{equation}
			Note $\nabla\mathcal{F}^{-1}=B^{-1}$. This, \eqref{transform-vector} and the chain rule give
			\begin{equation*}
			\nabla v=B^{-T}\widehat{\nabla}\widehat{v}B^{-1},
			\end{equation*}
			{and}
			\begin{equation*}
			\eps\left(v\right)=B^{-T}\widehat{\eps}\left(\widehat{v}\right)B^{-1}.
			\end{equation*}
			This relationship implies that the transform \eqref{transform-vector} sets up a one-to-one correspondence between $RM\left(\widehat{K}\right)$ and $RM\left(K\right)$. In general, the space $RM(K)$ is exactly the lowest order element of the first kind of N\'{e}d\'{e}lec element $ND_0(K)$ \cite[Eq.~(2.3.66)]{Boffi-Brezzi-Fortin2013}, and the transform \eqref{transform-vector} sets up a one-to-one correspondence between $ND_k\left(\widehat{K}\right)$ and $ND_k(K)$ for each $k\geqslant0$, see \cite[Remark 2.3.6]{Boffi-Brezzi-Fortin2013}.

			Given a differentiable function $\widehat{\tau}:\ \widehat{K}\rightarrow\Sym$, the matrix Piola transform in \cite[Eq.~(5.8)]{Johnson-Mercier} and \cite[Eq.~(4.3)]{Arnold-Winther-conforming} defines $\tau:\ K\rightarrow\Sym$  {as follows}
			\begin{equation}
			\tau(x):=B\widehat{\tau}(\widehat{x})B^T.
			\label{transform-matrix}
			\end{equation}
			Note that the divergence relationship  {\cite[Eq.~(4.4)]{Arnold-Winther-conforming} reads}
			\begin{equation}
			\dv\tau\left(x\right)=B\widehat{\dv}\widehat{\tau}\left(\widehat{x}\right).
			\label{eq: div}
			\end{equation}
			This implies that the transform \eqref{transform-matrix} sets up a one-to-one correspondence between $H\left(\widehat{\dv},\widehat{K};\Sym\right)$ and $H\left(\dv,K;\Sym\right)$.

%
%
%
%

			The following fact will be later used to characterize the divergence of macro-element bubble function spaces.

			\begin{remark}
				\label{rmk: RM-determined}
				A rigid motion function $v$ in $K$ can be determined by its value on an arbitrary $(n-1)$-dimensional face $F\subset K$.  {In fact,} the invariance of $RM(K)$ under covariant transforms implies that $F$ can be assumed to lie on the $x$-axis for the case $n=2$, or on the $x-y$ plane for the case $n=3$.  {In fact, suppose} that $v\in RM(K)$ vanishes on $F$. The expression \eqref{eq: RM} implies
				\begin{align*}
				&a_1=a_2=\cdots=a_n=0,\\
				&b_{ij}=0,\quad\forall 1\leqslant i<j\leqslant n.
				\end{align*}
				Therefore, any $v\in RM(K)$ that vanishes on $F$ must be zero.
		\end{remark}}
		


		\subsection{The mixed finite element spaces}
		\label{subsec: MFEspace}
		This subsection introduces some general notation for macro-elements and gives the definition of the mixed {finite} element spaces. More details for the macro-elements and the $H(\dv;\Sym)$-bubble function space on macro-elements will be given in Section \ref{sec: lower-order} for polynomial degree $2\leqslant k\leqslant n$, respectively.

		The discrete stress spaces on {macro-element meshes} take the following form
		\begin{equation}
		\begin{aligned}
		\Sigma_{k,h}&=\{\sigma\in \Sigma:\ \sigma=\sigma_c+\sigma_b,\ \sigma_c\in H^1(\Omega;\Sym),\\
		&\quad\quad\trace{\sigma_c}{K}\in P_k(K;\Sym),\  {\forall K\in\T_h,\ \trace{\sigma_b}{M}\in\Sigma_{M,k,b},} \ \forall M\in\M_h\}.
		\end{aligned}
		\end{equation}
		 Here {$\Sigma_{M,k,b}\subset H_0\left(\dv,M;\Sym\right)$} denotes a generic macro-element bubble function space  {consisting of some piecewise polynomials on $M$ of degree at most $k$}, and will be specifically clarified in Section~\ref{sec: lower-order} for  {$n=2,3$ and $2\leqslant k\leqslant n$}, respectively. 
		
		Recall  {the $H(\dv;\Sym)$-bubble function space $\Sigma_{K,k,b}$ on element $K$} from \cite{Hu2015trianghigh}:
		\begin{equation}
		\Sigma_{K,k,b}=\{\tau\in P_k(K;\Sym):\ \trace{\tau\bm{n}}{\partial K}=0\},
		\label{def: Hdiv-bubble}
		\end{equation}
		{which satisfies}
		\begin{equation}
		\dv\Sigma_{K,k,b}=RM^{\perp}\left(K\right).
		\label{eq: divbubble}
		\end{equation}
		Here $RM^{\perp}\left(K\right)$ is the {$L^2$} orthogonal complement space of $RM\left(K\right)$ with respect to $P_{k-1}\left(K;\R^n\right)$. At least, the space $\Sigma_{M,k,b}$ contains $\Sigma_{K,k,b}$ for all sub-elements {$K\subset M$}.

		The discrete displacement {space is}
		\begin{equation}
		V_{k,h}=\{v\in V:\ \trace{v}{K}\in P_{k-1}(K;\R^n),\ \forall K\in\T_h\}.
		\label{eq: V-space}
		\end{equation}

		{The mixed finite element approximation of \eqref{continuousP} reads: Find $\left(\sigma_h,u_h\right)\in\Sigma_{k,h}\times V_{k,h}$ such that
			\begin{equation}
			\left\{
			\begin{aligned}
			(A\sigma_h,\tau_h)_{\Omega}+(\dv\tau_h,u_h)_{\Omega}&=0,\\
			(\dv\sigma_h,v_h)_{\Omega}&=(f,v_h)_{\Omega}
			\end{aligned}
			\right.
			\label{discretizedP}
			\end{equation}
			holds for any $(\tau_h,v_h)\in\Sigma_{k,h}\times V_{k,h}$.}
		
		\subsection{Discrete stability}
		\label{subsec: stability}
		The discrete stability of {\eqref{discretizedP}} relies on the K-ellipticity and the discrete inf-sup condition \cite{Boffi-Brezzi-Fortin2013}. These two conditions read as follows:
		\begin{enumerate}
			\item \textbf{K-ellipticity condition:} There exists a constant $C>0$ such that
			\begin{equation}
			\left(A\sigma_h,\sigma_h\right)\geqslant C\norm{\sigma_h}_{H(\dv)}^2
			\label{cond: K-ellipticity}
			\end{equation}
			for any $\sigma_h\in Z_h:=\left\{\tau_h\in\Sigma_{k,h}:\ \left(\dv\tau_h,v_h\right)=0,\ \forall v_h\in V_{k,h}\right\}$.
			
			\item \textbf{Discrete inf-sup condition:} There exists a constant $C>0$ such that
			\begin{equation}
			{\inf_{v_h\in V_{k,h}}\sup_{\tau_h\in\Sigma_{k,h}}}\frac{\left(\dv\tau_h,v_h\right)}{\norm{\tau_h}_{H(\dv)}\norm{v_h}_0}\geqslant C.
			\label{cond: inf-sup}
			\end{equation}
		\end{enumerate}

		\section{Lower order elements for the 2D problem and the 3D problem}
		\label{sec: lower-order}
		
		In this section, {the} mixed elements of  {degree} $k=2$ for the 2D problem and  {degree} $k=3$ for the 3D problem are {introduced} on some macro-element meshes. Macro-element bubble function spaces are constructed by utilizing {the geometric structures} of the  {macro-elements}. Besides, another approach is presented to prove the {discrete stability} and optimal convergence of the  {$P_2-P_1$ element} of \cite{ChenHuang2022} on another macro-element mesh.

		\subsection{The mixed element of  {degree} $k=2$ for the 2D problem}
		\label{subsec: 2DP2-element}
		
		For {$n=2$}, each macro-element consists of four sub-elements, {see} Figure~\ref{fig:2dp2p1}. For each macro-element $M$, apart from the bubble functions {in \eqref{def: Hdiv-bubble} on each of the four sub-elements $K\subset M$,} six globally continuous functions vanishing on $\partial M$ and three new macro-element bubble functions are added to $\Sigma_{M,2,b}$.


		\subsubsection{The definition of a macro-element}
		\label{subsubsec: 2DP2-macroelement}
		
		See Figure~\ref{fig:2dp2p1} for an illustration of the macro-element. The macro-element $M=x_0x_1x_2$ consists of four  {sub-elements}, namely $K_1=x_1m_1m_2$, $K_2=x_0m_1m_2$, $K_3=x_0m_1m_3$, $K_4=x_2m_1m_3$. The four  {sub-elements} are separated by three interior edges, namely $E_1=m_1m_2$, $E_2=m_1x_0$ and $E_3=m_1m_3$. Here $m_i$ is the midpoint of the edge of $M$ for $i=1,2,3$, and $d_i$ is the midpoint of $E_i$ for $i=1,2,3$.

		\begin{figure}[h]
			\centering
			\begin{tikzpicture}[line width=0.5pt,scale=1]
			\coordinate (center) at (0,0);
			\coordinate (V1) at (2,3);
			\coordinate (V2) at (-2,0);
			\coordinate (V3) at (3,0);
			
			\coordinate (M1) at (0.5,0);
			\coordinate (M2) at (0,1.5);
			\coordinate (M3) at (2.5,1.5);
			\coordinate (C1) at (0.25,0.75);
			\coordinate (C2) at (1.25,1.5);
			\coordinate (C3) at (1.5,0.75);
			
			\coordinate (K1) at (-0.5,0.2);
			\coordinate (K2) at (0.83333,1.5);
			\coordinate (K3) at (1.66666,1.5);
			\coordinate (K4) at (2,0.2);

			
			\draw[fill=black] {(V1)} circle (1pt);
			\draw[fill=black] {(V2)} circle (1pt);
			\draw[fill=black] {(V3)} circle (1pt);

			\draw (V1) -- (V2) -- (V3) -- cycle;
			\draw (M2) -- (M1) -- (M3);
			\draw (M1) -- (V1);
			
			\draw[fill=black] {(M1)} circle (1pt);
			\draw[fill=black] {(M2)} circle (1pt);
			\draw[fill=black] {(M3)} circle (1pt);
			\draw[fill=black] {(C1)} circle (1pt);
			\draw[fill=black] {(C2)} circle (1pt);
			\draw[fill=black] {(C3)} circle (1pt);
			
			\node [above =0.1cm] at (V1) {$x_0$}; 
			\node [below =0.1cm] at (V2) {$x_1$}; 
			\node [below =0.1cm] at (V3) {$x_2$}; 
			\node [below right=0.1cm] at (M1) {$m_1$}; 
			\node [left =0.2cm] at (M2) {$m_2$}; 
			\node [right=0.02cm] at (M3) {$m_3$}; 
			\node [below =0.1cm] at (C1) {$d_1$}; 
			\node [below left =0.05cm] at (C2) {$d_2$}; 
			\node [below =0.05cm] at (C3) {$d_3$};

			\coordinate (E1) at (0.2,0.95);
			\coordinate (E2) at (1.4,1.8);
			\coordinate (E3) at (1.8,1);
			\node at (E1) {$E_1$}; 
			\node at (E2) {$E_2$}; 
			\node at (E3) {$E_3$};

			\node [above] at (K1) {$K_1$}; 
			\node [above] at (K2) {$K_2$}; 
			\node [above right] at (K3) {$K_3$}; 
			\node [above right] at (K4) {$K_4$}; 
			
			\end{tikzpicture}
			
			\caption{The macro-element for the mixed element of  {degree} $k=2$ in 2D.}
			\label{fig:2dp2p1}
		\end{figure}
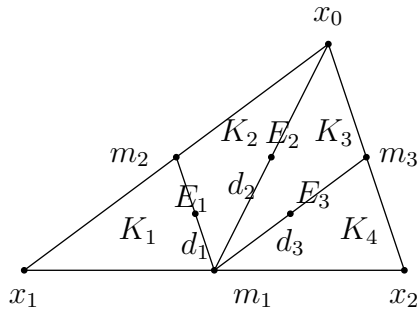

		\subsubsection{The macro-element bubble functions}
		\label{sec: 2Dmacrobubble}
		
		 {Recall the Lagrange basis function $\Phi_X$ defined in Section \ref{subsec: notation}.} The macro-element bubble functions fall into three categories:
		\begin{enumerate}
			\item Piecewise $H(\dv;\Sym)$-bubble functions: $\trace{\tau}{K}\in \Sigma_{K,2,b}$ for each $K\subset M$;
			\item Globally continuous functions that vanish on $\partial M$: $\Phi_{d_1}(\bm{t}\bm{n}^T+\bm{n}\bm{t}^T)$ and $\Phi_{d_1}\bm{n}\bm{n}^T$ with the tangential vector $\bm{t}$ and the normal vector $\bm{n}$ of the interior edge $E_1$. And four other bubble functions  {can be} defined for $d_2$ and $d_3$ in the similar manner;
			\item Globally discontinuous functions that have vanishing normal components on $\partial M$: $\left(\trace{\Phi_{m_1}}{M}\right)\bm{t}\bm{t}^T$ with the tangential vector $\bm{t}$ of the edge $x_1x_2$. And two other bubble functions can be defined for $m_2$ and $m_3$ in the similar manner.
		\end{enumerate}
		The linear span of the macro-element bubble functions {in (1)-(3)} defines the macro-element bubble function space $\Sigma_{M,2,b}$.

		\begin{remark}
			The macro-element bubble functions {in (1)-(2)} are contained in the $P_2$ analogous element of \cite{Hu2015trianghigh}. In particular, the macro-element bubble functions {in (2)} are associated with the DoFs in the interior of the edges. The macro-element bubble functions {in (3)} are proposed by utilizing {the collinear edges in} the macro-elements and the techniques {as} in \cite{HuMa}.
		\end{remark}
		
		\subsubsection{Discrete stability}
		
		Note that $\dv\Sigma_{2,h}\subseteq V_{2,h}$. This implies the $K$-ellipticity condition {\eqref{cond: K-ellipticity}}. Therefore it suffices to check that the discrete inf-sup condition {\eqref{cond: inf-sup}} holds. The discrete inf-sup condition is verified in two steps {as} in \cite[Theorem 2.2, Lemma 3.1]{Hu2015trianghigh}. {In the following,} Lemma \ref{Lem: interpolation2D} constructs an interpolation operator $I_h:H^1(\Omega;\Sym)\rightarrow\Sigma_{2,h}\cap H^1(\Omega;\Sym)$, {and} Lemma \ref{Lem: RMperp-2D} proves the {identity} $\dv\Sigma_{M,2,b}=RM^{\perp}(M)$.
		
		\begin{lemma}
			There exists an interpolation operator $I_h:H^1(\Omega;\Sym)\rightarrow\Sigma_{2,h}\cap H^1(\Omega;\Sym)$ such that for any macro-element edge $E$,
			\begin{equation}
			\int_{E}\big((\tau-I_h\tau)\bm{n}_E\big)\cdot w\dx{s}=0,\ \forall\tau\in H^1(\Omega;\Sym),\ \forall w\in RM(M^+\cup M^-).
			\label{req1}
			\end{equation}
			Here $\bm{n}_E$ denotes the unit normal vector of the edge $E$, $M^{+}$ and $M^-$ denote  the two macro-elements sharing $E$. {Besides, it holds}
			\begin{equation}
			\norm{I_h\tau}_{H(\dv)}\leqslant{C\norm{\tau}_{1}}.
			\label{ineq: Ih2D_stability}
			\end{equation}
			
			\label{Lem: interpolation2D}
		\end{lemma}
		\begin{proof}
			Let $I_1:H^1(\Omega;\Sym)\rightarrow\Sigma_{2,h}\cap H^1(\Omega;\Sym)$ be the Scott-Zhang interpolation \cite{ScottZhang} {with}
			\begin{equation}
			{\sum_{K\in\T_h}\left(\norm{\nabla\left(\tau-I_1\tau\right)}_{0,K}^2+h_K^{-2}\norm{\tau-I_1\tau}_{0,K}^2\right)\leqslant C\abs{\tau}_1^2,\quad\forall \tau\in H^1(\Omega;\Sym).}
			\label{ineq: ScottZhang}
			\end{equation}
			As in \cite[Lemma 3.1]{Hu2015trianghigh}, $I_h\tau$ can be obtained by modifying $I_1\tau$ with {some} macro-element edge bubble functions, {i.e. there exists $\delta\in\Sigma_{2,h}$ such that}
			\begin{equation}
			\int_{E}(\delta \bm{n}_E)\cdot w\dx{s}=\int_{E}\big((\tau-I_1\tau)\bm{n}_E\big)\cdot w\dx{s},\ \forall w\in RM(M^+\cup M^-).
			\label{eq: delta2D}
			\end{equation}
			Take {the} edge $E=x_0x_1$ as an example. Choose two quadratic Lagrange basis functions: $\phi_1=\lambda_{x_0}\lambda_{m_2}$ and $\phi_2=\lambda_{x_1}\lambda_{m_2}$, see Figure \ref{fig:2dp2p1}. Let {$\{T_{j}\}_{1\leqslant j\leqslant3}$} be the canonical basis of $\Sym$. Set
			\begin{equation*}
			{\delta=\sum_{j=1}^{3}\sum_{l=1}^{2}a_{j}^{(l)}T_{j}\phi_l.}
			\end{equation*} 
			{Since $\trace{w}{E}\in P_1\left(E\right)$}, it suffices to show that for any {$1\leqslant j\leqslant3$, $l=1,2$,} and given {$f_{j}\in L^2(E)$}, there exists {$a_{j}^{(l)}$} such that for any $p\in P_1(E)$,
			\begin{equation*}
			{\sum_{l=1}^{ {2}}a_{j}^{(l)}\int_{E}\phi_lp\dx{s}=\int_{E}f_{j}p\dx{s}.}
			\end{equation*}
			This is true since the adjoint system
			\begin{equation}
			{\int_{E}\phi_lp\dx{s}=0,\quad l=1,2}
			\label{eq: adjoint}
			\end{equation}
			admits only zero solution $p=0$ in $P_1(E)$. {Actually, \eqref{eq: adjoint} implies that $p$ vanishes at the midpoints of $x_0m_2$ and $x_1m_2$, and this shows $p=0$.}
			
			 {The cases of the other} edges can be similarly discussed. {This allows to define the modification $\delta$ to satisfy \eqref{eq: delta2D}.} The matrix Piola transform \eqref{transform-matrix}, the local trace inequality, and a standard scaling argument as in \cite{Arnold-Winther-conforming,Hu2015trianghigh} show that
			\begin{equation*}
			\sum_{K\in\T_h}\left(\norm{\nabla\delta}_{0,K}^2+h_K^{-2}\norm{\delta}_{0,K}^2\right)\leqslant
			C\sum_{K\in\T_h}\left(\norm{\nabla\left(\tau-I_1\tau\right)}_{0,K}^2+h_K^{-2}\norm{\tau-I_1\tau}_{0,K}^2\right).
			\end{equation*}
			This and the approximation property \eqref{ineq: ScottZhang} lead to
			\begin{equation}
			\sum_{K\in\T_h}\left(\norm{\nabla\delta}_{0,K}^2+h_K^{-2}\norm{\delta}_{0,K}^2\right)\leqslant C\abs{\tau}_1^2.
			\label{ineq: estimate-delta}
			\end{equation}
			Finally, \eqref{ineq: ScottZhang} and \eqref{ineq: estimate-delta} prove the bound \eqref{ineq: Ih2D_stability}.

		\end{proof}

		
		\begin{lemma}
			For the macro-element bubble function space {$\Sigma_{M,2,b}$}, it holds
			\begin{equation}
			\dv\Sigma_{M,2,b}=RM^{\perp}(M).
			\end{equation}
			\label{Lem: RMperp-2D}
		\end{lemma}
		\begin{proof}
			Since $\Sigma_{M,2,b}\subset H_0(\dv,M;\Sym)$, integration by parts shows $\dv\Sigma_{M,2,b}\subseteq RM^{\perp}(M)$. {It} suffices to show that: if for some $v\in RM^{\perp}(M)$,
			\begin{equation*}
			\int_{M}\dv\tau\cdot v\dx{x}=0,\ \forall\tau\in\Sigma_{M,2,b},
			\end{equation*}
			then it holds $v\in RM(M)$. Recall the $H(\dv)$ bubble functions in {\eqref{def: Hdiv-bubble}.} By choosing $\tau\in\Sigma_{K_j,2,b}$ for each element $K_j\subset M$ with $1\leqslant j\leqslant4$,  {the equation \eqref{eq: divbubble}} implies that $\trace{v}{K_j}\in RM(K_j)$. {Remark~\ref{rmk: RM-determined} shows} that a rigid motion function in $\R^2$ is determined by its value on an edge. It remains to use the macro-element bubble functions {in (2)-(3)} to deal with the jump {of $v$} across $E_i$ with $i=1,2,3$. {Let $\bm{t}$ be the tangential vector  {of} the edge $x_0x_1$. Let $\bm{t_i}$ and $\bm{n_i}$ be the tangential vector and the normal vector of the interior edge $E_i$, with $i=1,2,3$.} Integration by parts on each $K_j$ and  {the} summation over $1\leqslant j\leqslant 4$  {yields}
			\begin{equation*}
			0=\int_{M}\dv\tau\cdot v\dx{x}=\sum_{j=1}^{4}\int_{K_j}\dv\tau\cdot v\dx{x}=\sum_{i=1}^{3}\int_{E_i}(\tau \bm{n}_i)\cdot{\left(\jumpE{v}{E_i}\right)}\dx{s}.
			\end{equation*}
			 {The choices of} $\tau$ as $(\trace{\Phi_{m_2}}{M})\bm{t}\bm{t}^T$, $\Phi_{d_1}(\bm{t_1}\bm{n_1}^T+\bm{n_1}\bm{t_1}^T)$ and $\Phi_{d_1}\bm{n_1}\bm{n_1}^T$  {yield}
			\begin{equation}
			\left\{
			\begin{aligned}
			&\int_{E_1}\Phi_{d_1}{\left(\jumpE{v}{E_1}\right)}\dx{s}=0,\\
			&\int_{E_1}\Phi_{m_2}{\left(\jumpE{v}{E_1}\right)}\cdot\bm{t}\dx{s}=0.
			\end{aligned}
			\right.
			\label{eq: RMperp}
			\end{equation}
			Note  {that} ${\left(\jumpE{v}{E_1}\right)}\cdot\bm{t}\in P_1(E_1)$ and {$$\Span\left\{\trace{\Phi_{d_1}}{E_1},\trace{\Phi_{m_2}}{E_1}\right\}=\lambda_{m_2}P_1(E_1).$$}
				These facts imply ${\left(\jumpE{v}{E_1}\right)}\cdot\bm{t}=0$. Recall the expression given in \eqref{eq: RM}: $${\jumpE{v}{E_1}}=\begin{bmatrix}
				a_1+b_{12}y\\
				a_2-b_{12}x
				\end{bmatrix}.$$
				This shows that $b_{12}\left(\bm{t}\cdot\begin{bmatrix}
				y\\
				-x
				\end{bmatrix}\right)$ is constant on $E_1$. Since $x_0x_1$ and $E_1$ are not parallel, $\bm{t}\cdot\begin{bmatrix}
				y\\
				-x
				\end{bmatrix}$ is not a constant on $E_1$. This implies $b_{12}=0$ and ${\jumpE{v}{E_1}}$ is a constant vector. This and the first equation of \eqref{eq: RMperp} lead  {to} ${\jumpE{v}{E_1}}=0$. Similar arguments show ${\jumpE{v}{E_3}}=0$ and ${\jumpE{v}{E_2}}=0$. Therefore $v\in RM(M)$. This establishes $\dv\Sigma_{M,2,b}=RM^{\perp}(M)$.
		\end{proof}
		
		With the preparations above, the discrete inf-sup condition can be established.
		
		\begin{theorem}
			The discrete inf-sup condition holds, \emph{i.e.} there exists a positive constant $C$ independent of $h$ such that
			\begin{equation}
			\inf_{v_h\in V_{2,h}}\sup_{\tau_h\in\Sigma_{2,h}}\frac{(\dv\tau_h,v_h)}{\norm{\tau_h}_{H(\dv)}\norm{v_h}_{0}}\geqslant C.
			\label{eq: inf-sup-2D}
			\end{equation}
			\label{Thm: inf-sup-2D}
		\end{theorem}
		
		\begin{proof}
			The well-posedness of the continuous problem (see \cite{Arnold-Winther-conforming} for the 2D case) shows {that} for any $v_h\in V_h$, there exists $\tau_1\in H^1(\Omega;\Sym)$ such that
			\begin{equation*}
			\dv\tau_1=v_h,\ \norm{\tau_1}_{1}\leqslant C\norm{v_h}_{0}.
			\end{equation*}
			Let the interpolation operator $I_h$ be defined as in Lemma \ref{Lem: interpolation2D}. For any $w\in RM(M)$,  {an} integration by parts yields
			\begin{equation}
			\begin{aligned}
			&\int_{M}(v_h-\dv I_h\tau_1)\cdot w\dx{x}=\int_{M}\dv(\tau_1- I_h\tau_1)\cdot w\dx{x}\\
			&=-\int_{M}(\tau_1- I_h\tau_1):\eps(w)\dx{x}+\int_{\partial M}\big((\tau_1- I_h\tau_1)\bm{n}\big)\cdot w\dx{s}=0.
			\end{aligned}
			\end{equation}
			This implies $v_h-\dv I_h\tau_1\in RM^{\perp}(M)$. As in \cite[Theorem 3.1]{Hu2015trianghigh}, $\tau_2\in\Sigma_{2,h}$ can be chosen such that $\trace{\tau_2}{M}\in\Sigma_{M,2,b}$ for each macro-element $M$, and that
			\begin{equation}
			\begin{aligned}
			\dv\tau_2&=v_h-\dv I_h\tau_1,\\
			\norm{\tau_2}_{{H(\dv)}}&\leqslant C\norm{v_h-\dv I_h\tau_1}_{{0}}.
			\end{aligned}
			\label{tau2}
			\end{equation}
			Finally, let $\tau_h=I_h\tau_1+\tau_2$. This and \eqref{tau2} show that $\dv\tau_h=v_h$, and {the following estimate holds}
			\begin{equation}
			\norm{\tau_2}_{{H(\dv)}}\leqslant C\norm{v_h-\dv I_h\tau_1}_{{0}}\leqslant C\big(\norm{v_h}_{{0}}+\norm{I_h\tau_1}_{{H(\dv)}}\big).
			\label{ineq1}
			\end{equation}
			The stability of $I_h$ {shows}
			\begin{equation}
			\norm{I_h\tau_1}_{{H(\dv)}}\leqslant C\norm{\tau_1}_{{1}}\leqslant C\norm{v_h}_{{0}}.
			\label{ineq2}
			\end{equation}{The combination of} \eqref{ineq1} and \eqref{ineq2} yields
			\begin{equation}
			\norm{\tau_h}_{{H(\dv)}}\leqslant C\norm{v_h}_{{0}}.
			\end{equation}{This proves} the discrete inf-sup condition {\eqref{eq: inf-sup-2D}}.

		\end{proof}
		
		\subsubsection{{Error} estimates}
		
		By using the discrete inf-sup condition \eqref{eq: inf-sup-2D}, similar arguments as in \cite{Hu2015trianghigh,HuZhang2015tre} lead to the following error estimates for the pair  {$\Sigma_{2,h}\times V_{2,h}$.}
		
		\begin{theorem}
			Let $(\sigma,u)\in\Sigma\times V$ {solve \eqref{continuousP}} and $(\sigma_h,u_h)\in\Sigma_{2,h}\times V_{2,h}$ {solve \eqref{discretizedP}}. Then the following estimate holds
			\begin{equation}
			\norm{\sigma-\sigma_h}_{H(\dv)}+\norm{u-u_h}_{0}\leqslant Ch^2\left(\norm{\sigma}_{3}+\norm{u}_{2}\right),
			\end{equation}
			 {provided that $\sigma\in H^3\left(\Omega;\Sym\right)$ and $u\in H^2\left(\Omega;\R^2\right)$.}
			\label{Thm: 2Destimate}
		\end{theorem}
		\begin{proof}
			The standard theory of mixed
			finite element methods \cite{Boffi-Brezzi-Fortin2013} give the
			following quasi-optimal error estimate immediately
			\begin{equation}
			\norm{\sigma-\sigma_h}_{H(\dv)}+\norm{u-u_h}_{0}\leq C\inf\limits_{\tau_h\in\Sigma_{2,h},v_h\in V_{2,h}}\left(\norm{\sigma-\tau_h}_{H(\dv)}+\norm{u-v_h}_{0}\right),
			\end{equation}
			and the estimate follows from the standard approximation theory.
		\end{proof}

		 {Lemma~\ref{Lem: interpolation2D} and Lemma~\ref{Lem: RMperp-2D} imply that there exists an interpolation operator $\Pi_h:\ H^1\left(\Omega;\Sym\right)\rightarrow\Sigma_{2,h}$ such that
		\begin{equation*}
			\ip{\dv\left(\sigma-\Pi_h\sigma\right),v_h}{\Omega}=0,\quad\forall v_h\in V_{2,h}.
	\end{equation*}
	Similar arguments as in \cite{Arnold-Winther-conforming, Hu2015trianghigh} show {the following optimal convergence} for the stress in the $L^2$ norm.}
		
		\begin{theorem}
			{Suppose that $\sigma\in H^3(\Omega;\Sym)$. It holds}
			\begin{equation}
			\norm{\sigma-\sigma_h}_{0}\leqslant Ch^3\norm{\sigma}_{3}.
			\end{equation}
		\end{theorem}

		\subsection{The mixed  {element} of  {degree} $k=3$ for the 3D problem}

		Each macro-element for {$n=3$} also consists of four sub-elements, {see} Figure \ref{fig:3dp3p2}. For each macro-element $M$, apart from the bubble functions {in \eqref{def: Hdiv-bubble}} on each of the four sub-elements {$K\subset M$}, nine globally continuous functions vanishing on $\partial M$ and eight new macro-element bubble functions are added to $\Sigma_{M,3,b}$.

		\subsubsection{The definition of a macro-element}
		See Figure~\ref{fig:3dp3p2} for an illustration of the macro-element. The macro-element $M=x_0x_1x_2x_3$ consists of four  {tetrahedra}, namely $K_1=x_0x_1m_1m_2$, $K_2=x_1x_3m_1m_2$, $K_3=x_1x_3m_2m_3$ and $K_4=x_2x_3m_2m_3$. They are separated by 3 interior faces, namely $F_1=x_1m_1m_2$, $F_2=x_1x_3m_2$ and $F_3=x_3m_2m_3$. Here $m_i$ is the midpoint of the edge of $M$ for $i=1,2,3$,  {and $d_i$ is the centroid of the face $F_i$ for $1\leqslant i\leqslant3$. The four tetrahedra are generated by two successive uniform bisections of $M$. The new edges generated {by} the first bisection are $x_1m_2$ and $x_3m_2$, with the points of the trisection $g_3,g_4,g_5,g_6$. The new edges generated {by} the second bisection are $m_1m_2$ and $m_2m_3$, with the points of the trisection $g_1,g_2,g_7,g_8$.}

		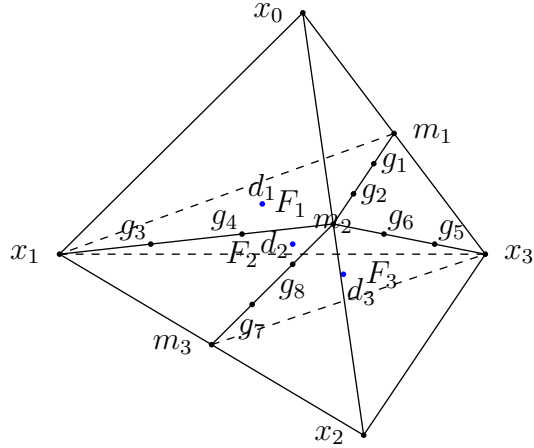
\begin{figure}[h]
			\centering
			\begin{tikzpicture}[line width=0.5pt,scale=0.8]
			\coordinate (center) at (0,0);
			\coordinate (V1) at (2,4);
			\coordinate (V2) at (-2,0);
			\coordinate (V3) at (3,-3);
			\coordinate (V4) at (5,0);
			
			\coordinate (M1) at (3.5,2);
			\coordinate (M2) at (2.5,0.5);
			\coordinate (M3) at (0.5,-1.5);

			%
			\coordinate (P1) at (3.166666,1.5);
			\coordinate (P2) at (2.8333333,1);	
			\coordinate (P3) at (-0.5,0.166666666);	
			\coordinate (P4) at (1,0.333333333);	
			\coordinate (P5) at (4.166666,0.16666666);
			\coordinate (P6) at (3.3333333,0.3333333);	
			\coordinate (P7) at (1.166666,-0.8333333);	
			\coordinate (P8) at (1.8333333,-0.16666666);	
			
			\coordinate (C1) at (1.333333333,0.8333333);
			\coordinate (C2) at (1.8333333333,0.1666666666);	
			\coordinate (C3) at (2.666666666,-0.3333333333);		
			
			\draw (V1) -- (V2) -- (V3) -- (V4) -- cycle;
			\draw (V1) -- (V3);
			\draw [dashed](V2) -- (V4);
			
			\draw (V2) -- (M2) -- (M3);
			\draw [dashed](M3) -- (V4);
			\draw (V4) -- (M2) -- (M1);
			\draw [dashed](M1) -- (V2);

			
			\draw[fill=black] {(V1)} circle (1pt);
			\draw[fill=black] {(V2)} circle (1pt);
			\draw[fill=black] {(V3)} circle (1pt);
			\draw[fill=black] {(V4)} circle (1pt);
			
			\draw[fill=black] {(M1)} circle (1pt);
			\draw[fill=black] {(M2)} circle (1pt);
			\draw[fill=black] {(M3)} circle (1pt);
			
			\draw[fill=black] {(P1)} circle (1pt);
			\draw[fill=black] {(P2)} circle (1pt);
			\draw[fill=black] {(P3)} circle (1pt);
			\draw[fill=black] {(P4)} circle (1pt);
			\draw[fill=black] {(P5)} circle (1pt);
			\draw[fill=black] {(P6)} circle (1pt);
			\draw[fill=black] {(P7)} circle (1pt);
			\draw[fill=black] {(P8)} circle (1pt);

			\draw[color=blue, fill=blue] {(C1)} circle (1pt);
			\draw[color=blue, fill=blue] {(C2)} circle (1pt);
			\draw[color=blue, fill=blue] {(C3)} circle (1pt);

			\node [left =0.1cm] at (V1) {$x_0$}; 
			\node [left =0.1cm] at (V2) {$x_1$}; 
			\node [left =0.1cm] at (V3) {$x_2$}; 
			\node [right =0.1cm] at (V4) {$x_3$}; 
			
			\node [right =0.1cm] at (M1) {$m_1$}; 
			\node at (M2) {$m_2$}; 
			\node [left =0.1cm] at (M3) {$m_3$}; 
			
			\node [right =-0.05cm] at (P1) {$g_1$}; 
			\node [right =-0.05cm] at (P2) {$g_2$}; 
			\node [above left =-0.1cm] at (P3) {$g_3$}; 
			\node [above left =-0.1cm] at (P4) {$g_4$}; 
			\node [above right =-0.1cm] at (P5) {$g_5$}; 
			\node [above right =-0.1cm] at (P6) {$g_6$}; 
			\node [below =0.05cm] at (P7) {$g_7$}; 
			\node [below =0.05cm] at (P8) {$g_8$}; 
			
			\node [above =-0.1cm] at (C1) {$d_1$}; 
			\node [left=-0.1cm] at (C2) {$d_2$}; 
			\node [below right=-0.1cm] at (C3) {$d_3$};

			\coordinate (F1) at (1.8,0.8333333);
			\coordinate (F2) at (1,0);	
			\coordinate (F3) at (3.3,-0.3333333333);		
			\node at (F1) {$F_1$}; 
			\node at (F2) {$F_2$}; 
			\node at (F3) {$F_3$}; 
			\end{tikzpicture}
			
			\caption{The macro-element for the mixed  {element} of  {degree} $k=3$ in 3D.}
			\label{fig:3dp3p2}
		\end{figure}

		\subsubsection{The macro-element bubble functions}  {Recall that $\Phi_X$ from Section \ref{subsec: notation} denotes the Lagrange basis function associated with $X$.} The macro-element bubble functions fall into three categories:
		\begin{enumerate}
			\item Piecewise $H(\dv;\Sym)$-bubble functions: $\trace{\tau}{K}\in \Sigma_{K,3,b}$ for each $K\subset M$;
			\item Globally continuous functions that vanish on $\partial M$: $\Phi_{d_1}(\bm{t}_1\bm{n}^T+\bm{n}\bm{t}_1^T)$, $\Phi_{d_1}(\bm{t}_2\bm{n}^T+\bm{n}\bm{t}_2^T)$ and $\Phi_{d_1}\bm{n}\bm{n}^T$. {Here} $\bm{t}_i$ {with $1\leqslant i\leqslant2$ are the independent tangential vectors} of the interior face $F_1$, with $i=1,2$, and $\bm{n}$ is the normal vector of $F_1$.  {And six} other bubble functions  {can be} defined for $d_2$ and $d_3$ in the similar manner;
			\item Globally discontinuous functions that have vanishing normal components on $\partial M$: $\left(\trace{\Phi_{g_3}}{M}\right)\bm{t}_2\bm{t}_2^T$ and $\left(\trace{\Phi_{g_3}}{M}\right)(\bm{t}_1\bm{t}_2^T+\bm{t}_2\bm{t}_1^T)$. {Here} $\bm{t}_1$ is the tangential vector  {of} the edge $x_1m_2$, $\bm{n}$ is the normal vector of $x_0x_1x_2$ and $\bm{t}_2=\bm{t}_1\times\bm{n}$.  {And six} other bubble functions  {can be} defined for $g_i$($4\leqslant i\leqslant6$) in the similar manner.
		\end{enumerate}
		The linear span of the macro-element bubble functions {in (1)-(3)} defines the macro-element bubble function space $\Sigma_{M,3,b}$.

		{\begin{remark}
				The macro-element bubble functions {in (1)-(2)} are contained in the $P_3$ analogous element of \cite{Hu2015trianghigh}. In particular, the macro-element bubble functions {in (2)} are associated with the DoFs in the interior of the faces. The macro-element bubble functions {in (3)} are proposed {due to} the coplanar faces in the macro-elements.
		\end{remark}}

		\subsubsection{Discrete stability}

		Similar to {$n=2$}, the following lemma uses {the DoFs on the macro-element faces} to modify the Scott-Zhang interpolation.
		
		\begin{lemma}
			There exists an interpolation operator $I_h:H^1(\Omega;\Sym)\rightarrow\Sigma_{3,h}\cap H^1(\Omega;\Sym)$ such that for any macro-element face $F$,
			\begin{equation}
			\int_{F}\big((\tau-I_h\tau)\bm{n}_F\big)\cdot w\dx{S}=0,\ \forall\tau\in H^1(\Omega;\Sym),\ \forall w\in RM(M^+\cup M^-).
			\label{eq: delta3D}
			\end{equation}
			Here $\bm{n}_F$ denotes the unit normal vector of the face $F$, $M^{+}$ and $M^-$ denote  the two macro-elements sharing $F$.  {Besides, it holds}
			\begin{equation}
			\norm{I_h\tau}_{H(\dv)}\leqslant{C\norm{\tau}_{1}}.
			\label{interpolation_stability}
			\end{equation}	
			\label{Lem: interpolation3DP3}
		\end{lemma}
		\begin{proof}
			Similar to {Lemma~\ref{Lem: interpolation2D} for $n=2$}, the modification on each macro-element face reduces to considering the  {corresponding} adjoint linear system. First consider the face $F=x_0x_1x_3$, see Figure~\ref{fig:3dinterpolation}. It suffices to show that if some $p\in P_1(F)$ satisfies
			{\begin{equation*}
				\int_{F}p\Phi_{X}\dx{S}=0,\quad\text{for $X=d_4,d_5,g_9,g_{10}$},
				\end{equation*}}then $p=0$. Here $d_4$ and $d_5$ are the centroid of the triangle $x_0x_1m_1$ and $x_1x_3m_1$, respectively, and $g_9$ and $g_{10}$ are the points of  {the} trisection for edge $x_1m_1$.
			
			The four $P_3$ Lagrange functions read $\Phi_{d_4}=27\lambda_{x_0}\lambda_{x_1}\lambda_{m_1}$, $\Phi_{d_5}=27\lambda_{x_1}\lambda_{x_3}\lambda_{m_1}$, $\Phi_{g_9}=\frac{27}{2}\lambda_{x_1}\lambda_{m_1}\left(\lambda_{x_1}-1\slash3\right)$ and $\Phi_{g_{10}}=\frac{27}{2}\lambda_{x_1}\lambda_{m_1}\left(\lambda_{m_1}-1\slash3\right)$.
			
			\begin{figure}[h]
				\centering
				\begin{tikzpicture}[line width=0.5pt,scale=0.7]
				\coordinate (center) at (0,0);
				\coordinate (V2) at (-3,0);    
				\coordinate (V1) at (1.5,4.5);        
				\coordinate (V4) at (4.5,0);      
				
				\coordinate (M1) at (3,2.25);
				
				\draw (V1) -- (V2) -- (V4) -- cycle;
				\draw (M1) -- (V2);
				
				\coordinate (C4) at (0.5,2.25);
				\coordinate (C5) at (1.5,0.75);
				\coordinate (P10) at (1,1.5);
				\coordinate (P9) at (-1,0.75);

				\draw[fill=black] {(V2)} circle (2pt);
				\draw[fill=black] {(V1)} circle (2pt);
				\draw[fill=black] {(V4)} circle (2pt);
				
				\draw[fill=black] {(M1)} circle (2pt);
				
				\draw[color=black,fill=black] {(C4)} circle (2pt);
				\draw[color=black,fill=black] {(C5)} circle (2pt);
				\draw[color=black,fill=black] {(P9)} circle (2pt);
				\draw[color=black,fill=black] {(P10)} circle (2pt);

				\node [above =0.01cm] at (V1) {$x_0$}; 
				\node [below =0.1cm] at (V2) {$x_1$}; 
				\node [below =0.1cm] at (V4) {$x_3$}; 
				\node [right =0.1cm] at (M1) {$m_1$}; 
				
				\node [above =0.01cm] at (C4) {$d_4$}; 
				\node [right =0.01cm] at (C5) {$d_5$}; 
				\node [above  =0.01cm] at (P9) {$g_9$}; 
				\node [above  =0.01cm] at (P10) {$g_{10}$}; 
				
				\end{tikzpicture}
				
				\caption{Interior Lagrange nodes of the macro-element face.}
				\label{fig:3dinterpolation}
			\end{figure}
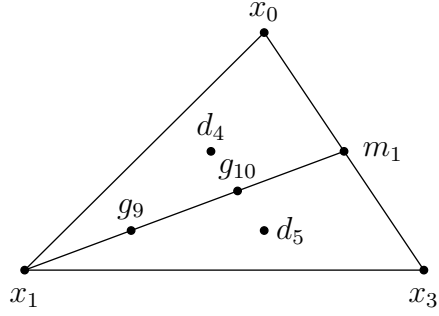
			
			Note that
			{\begin{equation*}
			\begin{aligned}
			\Span\left\{\trace{\Phi_{d_4}}{F},\trace{\Phi_{d_5}}{F},\trace{\Phi_{g_9}}{F},\trace{\Phi_{g_{10}}}{F}\right\}&=\lambda_{x_1}\lambda_{m_1}\Span\left\{\lambda_{x_0},\lambda_{x_3},\lambda_{x_1}-\frac{1}{3},\lambda_{m_1}-\frac{1}{3}\right\}\\
			&=\lambda_{x_1}\lambda_{m_1}\Span\left\{\lambda_{x_0},\lambda_{x_3},\lambda_{x_1},\lambda_{m_1}\right\},
			\end{aligned}
			\end{equation*}}
			and that $p\in P_1(F)\subseteq\Span\left\{\lambda_{x_0},\lambda_{x_3},\lambda_{x_1},\lambda_{m_1}\right\}$. This shows $\int_{F}\lambda_{x_1}\lambda_{m_1}p^2\dx{S}=0$ and $p=0$. 
			
			Moreover, for $F=x_0x_1x_2$, it suffices to use the face $P_3$ Lagrange functions supported on the triangle $x_1x_2m_2$. {This allows to define the modification $\delta$ to satisfy \eqref{eq: delta3D}.} Then a similar argument as in \cite{Arnold-Winther-conforming,Hu2015trianghigh} can verify the stability condition \eqref{interpolation_stability}.
			
		\end{proof}
		
		\begin{lemma}
			For the macro-element bubble function space defined in Section 3.3.2, it holds
			\begin{equation}
			\dv\Sigma_{M,3,b}=RM^{\perp}(M).
			\end{equation}
			\label{Lem: RMperp-3DP3}
		\end{lemma}
		\begin{proof}
			Similar to {Lemma~\ref{Lem: RMperp-2D} for $n=2$}, since $\dv\Sigma_{M,3,b}\subseteq RM^{\perp}(M)$, it suffices to show that for a piecewise rigid motion function $v$, the following equation
			\begin{equation}
			\sum_{i=1}^{3}\int_{F_i}(\tau \bm{n_i})\cdot{\jumpE{v}{F_i}}\dx{S}=0,\quad\forall\tau\in\Sigma_{M,3,b},
			\label{eq: divbubble-eq}
			\end{equation}
			implies ${\jumpE{v}{F_i}}=0$.
			
			Take {the face} $F_1=x_1m_1m_2$ for an example. {The choice of $\tau$ in \eqref{eq: divbubble-eq} by} the bubble functions associated with $g_1$, $g_2$ and $d_1$ yields
			\begin{equation}
			\left\{
			\begin{aligned}
			&\int_{F_1}\Phi_{d_1}{\jumpE{v}{F_1}}\dx{S}=0,\\
			&\int_{F_1}\Phi_{g_1}\left({\jumpE{v}{F_1}}\cdot \bm{t}\right)\dx{S}=\int_{F_1}\Phi_{g_2}\left({\jumpE{v}{F_1}}\cdot \bm{t}\right)\dx{S}=0\quad{\text{for all $\bm{t}$ parallel to $x_0m_1m_2$.}}
			\end{aligned}
			\right.
			\label{eq: RMperp2}
			\end{equation}
			
			Without loss of generality, assume $F_1$ {lies on the $x-y$ plane.} Let $\bm{t}=(t_1,t_2,t_3)^T$ be an arbitrary vector parallel to $x_0m_1m_2$. Note ${\jumpE{v}{F_1}}\cdot\bm{t}\in P_1(F_1)$ and {$$\Span\left\{\trace{\Phi_{d_1}}{F_1},\trace{\Phi_{g_1}}{F_1},\trace{\Phi_{g_2}}{F_1}\right\}=\lambda_{m_1}\lambda_{m_2}P_1(F_1).$$}{These facts} imply ${\jumpE{v}{F_1}}\cdot\bm{t}=0$ on $F_1$. {Recall the expression in \eqref{eq: RM}. Note  {that} $z=0$ holds on $F_1$. This leads  {to}
				\begin{equation*}
				\jumpE{v}{F_1}=\begin{bmatrix}
				a_1+b_{12}y\\
				a_2-b_{12}x\\
				a_3-b_{13}x-b_{23}y
				\end{bmatrix},
				\end{equation*}
				 {which} implies
				$${\jumpE{v}{F_1}}\cdot\bm{t}=\left(a_1t_1+a_2t_2+a_3t_3\right)-\left(b_{12}t_2+b_{13}t_3\right)x+\left(b_{12}t_1-b_{23}t_3\right)y,$$and
				\begin{equation*}
				\left\{
				\begin{aligned}
				b_{12}t_2+b_{13}t_3&=0,\\
				b_{12}t_1-b_{23}t_3&=0.
				\end{aligned}
				\right.
				\end{equation*}Note that $\bm{t}$ is an arbitrary vector parallel to $x_0m_1m_2$.}
			This means both {$(0,b_{12},b_{13})^T$ and $(b_{12},0,-b_{23})^T$} are orthogonal to $x_0m_1m_2$. {This} leads  {to} {$b_{12}=b_{13}=b_{23}=0$} since $x_0m_1m_2$ is not parallel to $F_1$. Therefore ${\jumpE{v}{F_1}}$ is a constant vector on $F_1$. This and the first equation of \eqref{eq: RMperp2} lead  {to} ${\jumpE{v}{F_1}}=0$. Similar arguments show ${\jumpE{v}{F_2}}=0$ and ${\jumpE{v}{F_3}}=0$.

		\end{proof}
		
		The discrete inf-sup condition is an immediate result of Lemma \ref{Lem: interpolation3DP3} and Lemma \ref{Lem: RMperp-3DP3}.
		\begin{theorem}
			The discrete inf-sup condition holds, \emph{i.e.} there exists a positive constant $C$ independent of $h$ such that
			\begin{equation}
			\inf_{v_h\in V_{3,h}}\sup_{\tau_h\in\Sigma_{3,h}}\frac{(\dv\tau_h,v_h)}{\norm{\tau_h}_{H(\dv)}\norm{v_h}_{0}}\geqslant C.
			\label{eq: inf-sup-3D}
			\end{equation}
		\end{theorem}
		
		\subsubsection{{Error} estimates}
		
		The discrete inf-sup condition \eqref{eq: inf-sup-3D} leads to the following error estimates for the pair  {$\Sigma_{3,h}\times V_{3,h}$.}
		
		\begin{theorem}
			Let $(\sigma,u)\in\Sigma\times V$ {solve \eqref{continuousP}} and $(\sigma_h,u_h)\in{\Sigma_{3,h}\times V_{3,h}}$ {solve \eqref{discretizedP}}. Then the following estimate holds
			\begin{equation}
			\norm{\sigma-\sigma_h}_{H(\dv)}+\norm{u-u_h}_{0}\leqslant Ch^3\left(\norm{\sigma}_{4}+\norm{u}_{3}\right),
			\end{equation}
			 {provided that $\sigma\in H^4\left(\Omega;\Sym\right)$ and $u\in H^3\left(\Omega;\R^3\right)$.}
			\label{Thm: 3DP3estimate}
		\end{theorem}

		 {Similar to {$n=2$}, Lemma~\ref{Lem: interpolation3DP3} and Lemma~\ref{Lem: RMperp-3DP3} lead to the following optimal {convergence} for the stress in the $L^2$ norm.}
		\begin{theorem}
			{Suppose that $\sigma\in H^4(\Omega;\Sym)$. It holds}
			\begin{equation}
			\norm{\sigma-\sigma_h}_{0}\leqslant Ch^4\norm{\sigma}_{4}.
			\end{equation}
		\end{theorem}

		{
			\subsection{The mixed element of  {degree} $k=2$ for the 3D problem.}
			\label{subsec: 3DP2}

			In this subsection, the mixed element of  {degree} $k=2$ is constructed on another macro-element mesh.  {For this case,} the complexity of the macro-element brings difficulty to the direct proof of  {$\dv\Sigma_{M,k,b}=RM^{\perp}\left(M\right)$}. Instead, another computer-assisted approach is presented. This approach reduces the proof to the case of a reference macro-element, which can be carried out by  {a} direct calculation of the rank of associated matrices corresponding to the bilinear form {$\ip{\widehat{\dv}\cdot,\cdot}{\widehat{M}}$.}

			\subsubsection{ {Transform} to a reference macro-element with the Piola transforms}
			
			{This subsection uses the transforms \eqref{transform-vector} and \eqref{transform-matrix} to reduce the proof of  {$\dv\Sigma_{M,k,b}=RM^{\perp}(M)$} to the case of a reference macro-element.
				
				Let $\widehat{M}$ be a reference macro-element and $\mathcal{F}:\ \widehat{M}\rightarrow M$ be an affine isomorphism that preserves the elements contained in $\widehat{M}$. {The map $\mathcal{F}$ can be viewed as the isomorphism defined in Section~\ref{subsec: Piola} when restricted on any $K\subset M$.} Note that the transform \eqref{transform-matrix} preserves the continuity of the normal component across the  {interior} faces {of $M$}. Therefore it also sets up a one-to-one correspondence between $H\left(\widehat{\dv},\widehat{M};\Sym\right)$ and $H\left(\dv,M;\Sym\right)$.}

			The following lemma states that, under certain assumptions of the macro-element bubble function space  {$\Sigma_{M,k,b}$}, one only needs to check that  {$\dv\Sigma_{\widehat{M},k,b}=RM^{\perp}\left(\widehat{M}\right)$} on a reference macro-element $\widehat{M}$.

			\begin{lemma}
				\label{Lem: reference}
				Let  {$\Sigma_{\widehat{M},k,b}$ and $\Sigma_{M,k,b}$} be the bubble function spaces on the macro-elements $\widehat{M}$ and $M$, respectively. Assume that the transform given by \eqref{transform-matrix} sets up a one-to-one correspondence between  {$\Sigma_{\widehat{M},k,b}$ and $\Sigma_{M,k,b}$}. Then the identity 
				 {\begin{equation}
				\dv\Sigma_{M,k,b}=RM^{\perp}\left(M\right)
				\label{eq: div-RM-K}
				\end{equation}
				holds if and only if
				\begin{equation}
				\widehat{\dv}\Sigma_{\widehat{M},k,b}=RM^{\perp}\left(\widehat{M}\right).
				\label{eq: div-RM-Khat}
				\end{equation}}
			\end{lemma}
			\begin{proof}
				It suffices to show that \eqref{eq: div-RM-Khat} implies \eqref{eq: div-RM-K}, since the converse can be proved {analogously}.
				
				{Note that $\Sigma_{M,k,b}\subset H_0\left(\dv,M;\Sym\right)$ implies $\dv\Sigma_{M,k,b}\subseteq RM^{\perp}\left(M\right)$.} It suffices to show that for a {$v\in V_M$}, the following equation
				 {\begin{equation}
				\int_{M}\dv\tau\cdot v\dx{x}=0,\quad\forall\tau\in\Sigma_{M,k,b}
				\end{equation}}
				implies $v\in RM\left(M\right)$.
				
				 {Recall that the transform given by \eqref{transform-vector} sets up a one-to-one correspondence between $RM\left(\widehat{K}\right)$ and $RM\left(K\right)$.} Let ${{\widehat{v}\in V_{\widehat{M}}}}$ be given by the inverse transform of \eqref{transform-vector}. For any  {$\widehat{\tau}\in\Sigma_{\widehat{M},k,b}$}, denote its Piola transform by  {$\tau\in\Sigma_{M,k,b}$}. Since
				\begin{align*}
				\int_{\widehat{M}}\widehat{\dv}\widehat{\tau}\cdot\widehat{v}\dx{\widehat{x}}&=\int_{M}\frac{1}{\det B}\left(B^{-1}\dv\tau\right)\cdot\left(B^Tv\right)\dx{x}=\frac{1}{\det B}\int_{M}\dv\tau\cdot v\dx{x}=0,
				\end{align*}
				this {and \eqref{eq: div-RM-Khat} imply} $\widehat{v}\in RM(\widehat{M})$ and $v\in RM(M)$. This concludes the proof.
				
			\end{proof}

			Let $\{\psi_j\}_{j=1}^{N_{\sigma}}$ be a basis of  {$\Sigma_{M,k,b}$} and $\{\zeta_i\}_{i=1}^{N_u}$ be a basis of $V_M$. The divergence matrix $B_{\dv}=\left(B_{\dv,ij}\right)_{ij}$ is defined by
			\begin{equation*}
			B_{\dv,ij}=\int_{M}\dv\psi_j\cdot\zeta_i\dx{x}.
			\end{equation*}
			Clearly, note that the vanishing normal component of bubble functions on $\partial M$ leads to the trivial inclusion  {$\dv\Sigma_{M,k,b}\subseteq RM^{\perp}\left(M\right)$, \eqref{eq: div-RM-K}} holds if and only if
			\begin{equation}
			\operatorname{rank}B_{\dv}=N_{u}-\dim RM\left(M\right)= {N_{u}-6}.
			\label{eq: rank}
			\end{equation}

			\subsubsection{The definition of a macro-element}

			See Figure~\ref{fig:3dp2p1} for an illustration of the macro-element. Note that $m_5$ is the centroid of $M=x_0x_1x_2x_3$, $m_i$ is the centroid of the face of $M$, with $1\leqslant i\leqslant 4$, and $g_i$ is the  {midpoint} of the interior edges of $M$, with $1\leqslant i\leqslant 8$.

			\begin{figure}[h]
				\centering
				\begin{tikzpicture}[line width=0.5pt,scale=0.9]
				\coordinate (center) at (0,0);
				\coordinate (x0) at (3,5);
				\coordinate (x1) at (-2,0);
				\coordinate (x2) at (2,-3);
				\coordinate (x3) at (5,0);
				\coordinate (m1) at (1.6666666667,-1);
				\coordinate (m2) at (3.33333333333,0.6666666667);
				\coordinate (m3) at (2,1.6666666667);
				\coordinate (m4) at (1,0.6666666667);
				\coordinate (m5) at (2,0.5);
				\coordinate (p1) at (2.5,2.75);
				\coordinate (p2) at (0,0.25);
				\coordinate (p3) at (2,-1.25);
				\coordinate (p4) at (3.5,0.25);
				\coordinate (p5) at (1.83333333,-0.25);
				\coordinate (p6) at (2.66666667,0.583333333);
				\coordinate (p7) at (2,1.08333333333333);
				\coordinate (p8) at (1.5,0.583333333333);

				\draw (x0) -- (x1) -- (x2) -- cycle;
				\draw (x0) -- (x3) -- (x2) -- cycle;
				\draw [dashed] (x1) -- (x3);

				\draw [dashed] (x0) -- (m5) -- (m1);
				\draw [dashed] (m5) -- (x1) -- (m1);
				\draw [dashed] (m5) -- (x2) -- (m1);
				\draw [dashed] (m5) -- (x3) -- (m1);
				\draw [dashed] (m5) -- (m2);
				\draw [dashed] (m5) -- (m3);
				\draw [dashed] (m5) -- (m4);
				
				\draw (x2) -- (m2) -- (x0);
				\draw (m2) -- (x3);
				
				\draw (x2) -- (m4) -- (x0);
				\draw (m4) -- (x1);
				
				\draw [dashed] (x1) -- (m3) -- (x0);
				\draw [dashed] (m3) -- (x3);	
				
				\draw[fill=black] {(x0)} circle (1pt);
				\draw[fill=black] {(x1)} circle (1pt);
				\draw[fill=black] {(x2)} circle (1pt);
				\draw[fill=black] {(x3)} circle (1pt);

				\draw[fill=black] {(m1)} circle (1pt);
				\draw[fill=black] {(m2)} circle (1pt);
				\draw[fill=black] {(m3)} circle (1pt);
				\draw[fill=black] {(m4)} circle (1pt);	
				\draw[fill=black] {(m5)} circle (1pt);

				\draw[fill=black] {(p1)} circle (1pt);
				\draw[fill=black] {(p2)} circle (1pt);
				\draw[fill=black] {(p3)} circle (1pt);
				\draw[fill=black] {(p4)} circle (1pt);
				
				\draw[fill=black] {(p5)} circle (1pt);
				\draw[fill=black] {(p6)} circle (1pt);
				\draw[fill=black] {(p7)} circle (1pt);
				\draw[fill=black] {(p8)} circle (1pt);

				\node [right =0.1cm] at (x0) {$x_0$}; 
				\node [left =0.1cm] at (x1) {$x_1$}; 
				\node [right =0.1cm] at (x2) {$x_2$}; 
				\node [right =0.1cm] at (x3) {$x_3$}; 
				
				\node [below left =0.01cm] at (m1) {$m_1$}; 
				\node [above right =0.01cm] at (m2) {$m_2$}; 
				\node [above left =0.01cm] at (m3) {$m_3$}; 
				\node [above left =-0.1cm] at (m4) {$m_4$}; 
				\node [below = 0.01cm] at (m5) {$m_5$};

				\node [left =-0.01cm] at (p1) {$g_1$}; 
				\node [below =0.01cm] at (p2) {$g_2$}; 
				\node [right =-0.01cm] at (p3) {$g_3$}; 
				\node [below =-0.01cm] at (p4) {$g_4$}; 
				\node [left =-0.01cm] at (p5) {$g_5$}; 
				\node [above right =-0.01cm] at (p6) {$g_6$}; 
				\node [left =-0.05cm] at (p7) {$g_7$}; 
				\node [above =-0.05cm] at (p8) {$g_8$};

				\end{tikzpicture}
				
				\caption{The macro-element for the mixed  {element} of  {degree} $k=2$ in 3D.}
				\label{fig:3dp2p1}
			\end{figure}
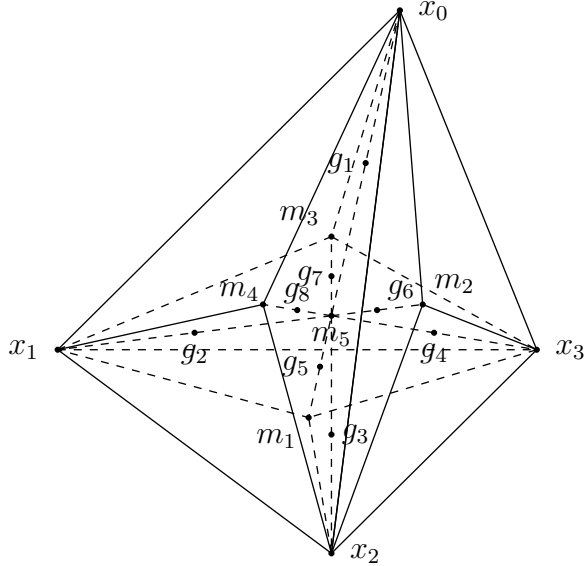
			
			Let $\Sigma_{CH}(M)$ be the $P_2$ analogy of the stress  {element} in \cite{ChenHuang2022} on the macro-element $M$,  {i.e.
		\begin{equation*}
\begin{aligned}
\Sigma_{CH}\left(M\right)=&\big\{\tau\in H\left(\dv,M;\Sym\right):\ \trace{\tau}{K}\in P_2\left(K;\Sym\right),\ \forall K\subset M,\ \text{$\tau$ is con-}\\
&\quad\text{tinuous at the nine vertices  {in $M$}, {$\bm{n}_i^T\tau\bm{n}_j$} is continuous  {on} the}\\
&\quad\text{twenty-six edges  {in $M$}},\ 1\leqslant i\leqslant j\leqslant 2\big\}.
\end{aligned}
\end{equation*}
Here $\bm{n}_i$ is the unit normal vector of the edges of the sub-elements of $M$.} The macro-element bubble function space is 
			{\begin{equation}
			\begin{aligned}
			\Sigma_{M,2,b}&=\big\{\tau\in\Sigma_{CH}(M)\cap H_0\left(\dv,M;\Sym\right):\ \text{$\tau$ vanishes at $m_i$, with $1\leqslant i\leqslant4$,}\\
			&\quad\quad\text{$\bm{n}_i^T\tau\bm{n}_j$ vanishes at the midpoint of the edges in the interior of each}\\
			&\quad\quad\text{face of $M$, with $1\leqslant i\leqslant j\leqslant2$}\big\}.
			\end{aligned}
			\label{eq: 3DP2bubble}
			\end{equation}}Note that the matrix Piola transform \eqref{transform-matrix} preserves all the properties {in \eqref{eq: 3DP2bubble}} and sets up a one-to-one correspondence between $\Sigma_{\widehat{M},2,b}$ and $\Sigma_{M,2,b}$. With Lemma \ref{Lem: reference}, take the {reference} macro-element to be the tetrahedron with vertices $\left(0,0,0\right)^T$, $\left(1,0,0\right)^T$, $\left(0,1,0\right)^T$ and $\left(0,0,1\right)^T$. Direct calculations in MATLAB show $N_{u}=12\times12=144$ and $\operatorname{rank}\left(B_{\dv}\right)=138$. This implies that \eqref{eq: rank} holds, and $\dv\Sigma_{M,2,b}=RM^{\perp}\left(M\right)$.

			The following lemma uses the  DoFs on the macro-element  {faces} to modify the Scott-Zhang interpolation.
			
			\begin{lemma}
				There exists an interpolation operator $I_h:H^1(\Omega;\Sym)\rightarrow\Sigma_{2,h}\cap H^1(\Omega;\Sym)$ such that for any macro-element face $F$,
				\begin{equation}
				\int_{F}\big((\tau-I_h\tau)\bm{n}_F\big)\cdot w\dx{S}=0,\ \forall\tau\in H^1(\Omega;\Sym),\ \forall w\in RM(M^+\cup M^-).
				\end{equation}
				Here {$\bm{n}_F$ denotes the unit normal vector of the face $F$, $M^{+}$ and $M^-$ denote  the two macro-elements sharing $F$.} {Besides, it holds}
				\begin{equation}
				\norm{I_h\tau}_{H(\dv)}\leqslant{C\norm{\tau}_{1}}.
				\label{eq: interpolation-stability3DP2}
				\end{equation}	
				\label{Lem: interpolation3DP2}
			\end{lemma}
			\begin{proof}
				Similar to the previous discussion for the $P_3$ element for the 3D problem, the modification on each macro-element face reduces to considering the adjoint linear system. First consider the face $F=x_0x_1x_2$, see Figure~\ref{fig:3dinterpolationp2}. It suffices to show that if some $p\in P_1(F)$ satisfies 
				\begin{equation*}
				\int_{F}p\Phi_{X}\dx{S}=0,\quad\text{for $X=g_9,g_{10},g_{11},m_{4}$,}
				\end{equation*}
				then $p=0$. Here $g_9$, $g_{10}$ and $g_{11}$ are the midpoints of  {the edges} $x_0m_4$, $x_1m_4$ and $x_2m_4$, respectively.
				
				The four $P_2$ Lagrange functions read $\Phi_{g_9}=4\lambda_{x_0}\lambda_{m_4}$, $\Phi_{g_{10}}=4\lambda_{x_1}\lambda_{m_4}$, $\Phi_{g_{11}}=4\lambda_{x_2}\lambda_{m_4}$ and $\Phi_{m_4}=\lambda_{m_4}\left(2\lambda_{m_4}-1\right)$.
				
				\begin{figure}[h]
					\centering
					\begin{tikzpicture}[line width=0.5pt,scale=0.7]
					\coordinate (center) at (0,0);
					\coordinate (x0) at (0,{3*sqrt(3)});
					\coordinate (x1) at (-3,0);
					\coordinate (x2) at (3,0);
					\coordinate (m) at (0,{sqrt(3)});

					\draw (x0) -- (x1) -- (x2) -- cycle;
					\draw (x0) -- (m) -- (x2);
					\draw (m) -- (x1);

					\coordinate (g0) at (0,{2*sqrt(3)});
					\coordinate (g1) at (-1.5,{sqrt(3)/2});
					\coordinate (g2) at (1.5,{sqrt(3)/2});
					
					\draw[fill=black] {(x0)} circle (2pt);
					\draw[fill=black] {(x1)} circle (2pt);
					\draw[fill=black] {(x2)} circle (2pt);
					\draw[fill=black] {(m)} circle (2pt);
					
					\draw[fill=blue] {(g0)} circle (2pt);
					\draw[fill=blue] {(g1)} circle (2pt);
					\draw[fill=blue] {(g2)} circle (2pt);

					\node [right =0.01cm] at (x0) {$x_0$}; 
					\node [left =0.01cm] at (x1) {$x_1$}; 
					\node [right =0.01cm] at (x2) {$x_2$}; 
					\node [below =0.1cm] at (m) {$m_4$};

					\node [right =0.01cm] at (g0) {$g_9$}; 
					\node [left =0.01cm] at (g1) {$g_{10}$}; 
					\node [right =0.01cm] at (g2) {$g_{11}$};

					\end{tikzpicture}
					
					\caption{Interior Lagrange nodes of the macro-element face.}
					\label{fig:3dinterpolationp2}
				\end{figure}
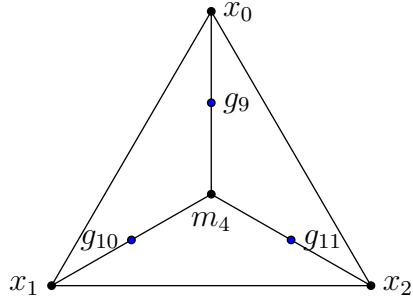
				Note that
				{\begin{equation*}
				\begin{aligned}
				\Span\left\{\trace{\Phi_{g_9}}{F},\trace{\Phi_{g_{10}}}{F},\trace{\Phi_{g_{11}}}{F},\trace{\Phi_{m_4}}{F}\right\}&=\lambda_{m_4}\Span\left\{\lambda_{x_0},\lambda_{x_1},\lambda_{x_2},\lambda_{m_4}-\frac{1}{2}\right\}\\
				&=\lambda_{m_4}\Span\left\{\lambda_{x_0},\lambda_{x_1},\lambda_{x_2},\lambda_{m_4}\right\}
				\end{aligned}
				\end{equation*}}and that $p\in P_1(F)\subseteq\Span\left\{\lambda_{x_0},\lambda_{x_1},\lambda_{x_2},\lambda_{m_4}\right\}$. This shows $\int_{F}\lambda_{m_4}p^2\dx{S}=0$ and $p=0$. 
				
				A similar argument as in \cite{Arnold-Winther-conforming,Hu2015trianghigh} can verify the stability condition \eqref{eq: interpolation-stability3DP2}.
				
			\end{proof}

			{Lemma \ref{Lem: reference} and the calculation of $\operatorname{rank}\left(B_{\dv}\right)$ lead to the following result.}
			\begin{lemma}
				For the macro-element bubble function space previously defined, it holds
				\begin{equation}
				\dv\Sigma_{M,2,b}=RM^{\perp}(M).
				\end{equation}
				\label{Lem: RMperp-3DP2}
			\end{lemma}

			The discrete inf-sup condition is an immediate result of Lemma \ref{Lem: interpolation3DP2} and Lemma \ref{Lem: RMperp-3DP2}.
			\begin{theorem}
				The discrete inf-sup condition holds, \emph{i.e.} there exists a positive constant $C$ independent of $h$ such that
				\begin{equation}
				\inf_{v_h\in V_{2,h}}\sup_{\tau_h\in\Sigma_{2,h}}\frac{(\dv\tau_h,v_h)}{\norm{\tau_h}_{H(\dv)}\norm{v_h}_{0}}\geqslant C.
				\label{eq: inf-sup-3DP2}
				\end{equation}
			\end{theorem}
			
			\subsubsection{The error estimates}
			The discrete inf-sup condition \eqref{eq: inf-sup-3DP2} leads to the following error estimates for the pair  {$\Sigma_{2,h}\times V_{2,h}$}.
			
			\begin{theorem}
				Let $(\sigma,u)\in\Sigma\times V$ {solve \eqref{continuousP}} and $(\sigma_h,u_h)\in\Sigma_{2,h}\times V_{2,h}$ {solve \eqref{discretizedP}}. Then the following estimate holds
				\begin{equation}
				\norm{\sigma-\sigma_h}_{H(\dv)}+\norm{u-u_h}_{0}\leqslant Ch^2\left(\norm{\sigma}_{3}+\norm{u}_{2}\right),
				\end{equation}
				 {provided that $\sigma\in H^3\left(\Omega;\Sym\right)$ and $u\in H^2\left(\Omega;\R^3\right)$.}
				\label{Thm: 3DP2estimate}
			\end{theorem}

		 {Similar to the previous elements, Lemma~\ref{Lem: interpolation3DP2} and Lemma~\ref{Lem: RMperp-3DP2} lead to the following optimal error estimate for the stress in the $L^2$ norm.}
			\begin{theorem}
				The following estimate holds:
				\begin{equation}
				\norm{\sigma-\sigma_h}_{0}\leqslant Ch^3\norm{\sigma}_{3},
				\end{equation}
				provided that $\sigma\in H^3(\Omega;\Sym)$.
			\end{theorem}

		}

		\section{An exact sequence in two dimensions}

		This section constructs an $H^2$ conforming composite finite element space in 2D. A unisolvent set of DoFs is given to determine inter-macroelement continuity {conditions}. This element together with the {symmetric} $H(\dv)$-conforming finite element in {Section~\ref{subsec: 2DP2-element}} form a discrete elasticity sequence in 2D.
		
		\subsection{An $H^2$ conforming finite element space}
		
		Let $M=\cup_{i=1}^4K_i$ be a macro-element from {Section~\ref{subsubsec: 2DP2-macroelement}}. Let $x_i$ be the vertex of $M$, $m_i$ be the midpoint of the edge of $M$, and $d_i$ be the midpoint of  {the interior edge $E_i$}, with $1\leqslant i\leqslant3$, see Figure \ref{fig:2dH2}. Let $\bm{t}$ and $\bm{n}$ be the tangential vector and the normal vector of the edge of the macro-element. The shape function space of the $H^2$ finite element on $M$ is piecewise quartic polynomials satisfying the following continuity {conditions}:
		\begin{equation}
			\begin{aligned}
			{U_M}:=&\big\{u\in L^2\left(M\right):\ \trace{u}{K_i}\in P_4(K_i),\ 1\leqslant i\leqslant 4,\ \text{$u$ is $C^2$ at $x_0$},\\
			&\quad\text{$u$, $\nabla u$, $\partial_{\bm{tt}}^2u$ and $\partial_{\bm{tn}}^2u$ are continuous at $m_i$, $i=1,2,3$,}\\
			&\quad\text{$u$ is continuous at $d_i$, $i=1,3$}\big\}.
			\end{aligned}
			\end{equation}

		\begin{figure}[h]
			\centering
			\begin{tikzpicture}[line width=0.5pt,scale=1]
			\coordinate (center) at (0,0);
			\coordinate (V1) at (2,3);
			\coordinate (V2) at (-2,0);
			\coordinate (V3) at (3,0);
			
			\coordinate (M1) at (0.5,0);
			\coordinate (M2) at (0,1.5);
			\coordinate (M3) at (2.5,1.5);
			
			\coordinate (C1) at (0.25,0.75);
			\coordinate (C2) at (1.25,1.5);
			\coordinate (C3) at (1.5,0.75);

			\draw[fill=white] {(V1)} circle (6pt);
			\draw[fill=white] {(V1)} circle (4pt);
			\draw[fill=black] {(V1)} circle (2pt);
			\draw[fill=white] {(V2)} circle (6pt);
			\draw[fill=white] {(V2)} circle (4pt);
			\draw[fill=black] {(V2)} circle (2pt);
			\draw[fill=white] {(V3)} circle (6pt);
			\draw[fill=white] {(V3)} circle (4pt);
			\draw[fill=black] {(V3)} circle (2pt);
			
			
			\draw[thick,->, >=latex] (M1) -- ($(M1)!0.6cm!90:(V2)$);				\draw[thick,->, >=latex] (M1) -- ($(M1)!0.4cm!90:(V2)$);
			\draw[thick,->, >=latex] (M1) -- ($(M1)!0.6cm!0:(V2)$);
			
			\draw[thick,->, >=latex] (M3) -- ($(M3)!0.6cm!90:(V3)$);				\draw[thick,->, >=latex] (M3) -- ($(M3)!0.4cm!90:(V3)$);
			\draw[thick,->, >=latex] (M3) -- ($(M3)!0.6cm!0:(V3)$);
			
			\draw[thick,->, >=latex] (M2) -- ($(M2)!0.6cm!90:(V1)$);				\draw[thick,->, >=latex] (M2) -- ($(M2)!0.4cm!90:(V1)$);
			\draw[thick,->, >=latex] (M2) -- ($(M2)!0.6cm!0:(V1)$);

			\draw (V1) -- (V2) -- (V3) -- cycle;
			\draw (M2) -- (M1) -- (M3);
			\draw (M1) -- (V1);
			
			\draw[fill=black] {(M1)} circle (1pt);
			\draw[fill=black] {(M2)} circle (1pt);
			\draw[fill=black] {(M3)} circle (1pt);
			\draw[color=black, fill=black] {(C1)} circle (1pt);
			\draw[fill=black] {(C2)} circle (1pt);
			\draw[color=black, fill=black] {(C3)} circle (1pt);
			
			\node [above =0.1cm] at (V1) {$x_0$}; 
			\node [below =0.15cm] at (V2) {$x_1$}; 
			\node [below =0.15cm] at (V3) {$x_2$}; 
			\node [below right=0.15cm] at (M1) {$m_1$}; 
			\node [left =0.2cm] at (M2) {$m_2$}; 
			\node [below right=0.02cm] at (M3) {$m_3$}; 
			\node [below =0.1cm] at (C1) {$d_1$}; 
			\node [below left =0.05cm] at (C2) {$d_2$}; 
			\node [below =0.05cm] at (C3) {$d_3$};

			\coordinate (K1) at (-0.5,0.2);
			\coordinate (K2) at (0.83333,1.5);
			\coordinate (K3) at (1.66666,1.5);
			\coordinate (K4) at (2,0.2);
			\node [above] at (K1) {$K_1$}; 
			\node [above] at (K2) {$K_2$}; 
			\node [above right] at (K3) {$K_3$}; 
			\node [above right] at (K4) {$K_4$};

			\coordinate (E1) at (0.2,0.95);
			\coordinate (E2) at (1.4,1.8);
			\coordinate (E3) at (1.8,1);
			\node at (E1) {$E_1$}; 
			\node at (E2) {$E_2$}; 
			\node at (E3) {$E_3$};

			\end{tikzpicture}
			
			\caption{Twenty-seven DoFs for each macro-element.}
			\label{fig:2dH2}
		\end{figure}
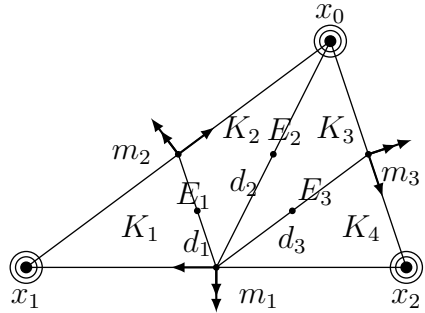
		
		The local DoFs {of $U_M$} are given as follows:
		\begin{enumerate}
			\item The value of $u$, $\nabla u$, $\nabla^2u$ at all the vertices of $M$.
			\item The value of $\partial_{\bm{n}}u$, $\partial_{\bm{tn}}^2u$, $\partial_{\bm{tt}}^2u$ at all the midpoints of edges of $M$.
		\end{enumerate}

		The following lemma states that, on each macro-element $M$, the shape function space {${U_M}$} belongs to $C^1(M)$.
		\begin{lemma}
			It holds ${U_M}\subset C^1(M)$.
			\label{C1}
		\end{lemma}
		\begin{proof}
			For any $u\in {U_M}$, {$u$} is $C^2$ at $x_0$, $C^1$ at $m_i$ with $i=1,2,3$, and $C^0$ at $d_i$ with $i=1,3$. These facts imply $u$ is continuous across {the edge} $E_i$ with $i=1,2,3$. Therefore $u\in C(M)$. 
			
			Let $\bm{l}$ be the tangential vector  {of} the edge $m_1m_2$. {Note that $$\partial_{\bm{ll}}^2u=\nabla^2u:\bm{ll}^T.$$This and} the continuity of $u$ across $m_1m_2$ imply $\nabla^2u:\bm{ll}^T$ is continuous at $m_2$. Note that $\bm{ll}^T,\bm{tt}^T,\bm{tn}^T+\bm{nt}^T$ form a basis of $\Sym$. This shows $\nabla^2u$ is continuous at $m_2$. Similar arguments prove the cases for $m_1$ and $m_3$. Therefore $\nabla^2u$ is continuous at $m_i$ for $i=1,2,3$. 
			
			The $C^2$ continuity of $u$ at $m_1$ and $m_2$ implies that $\nabla u$ is continuous across {the edge} $m_1m_2$. Similar arguments prove the cases for $m_1x_0$ and $m_1m_3$. This concludes the proof.
		\end{proof}
		
		The following theorem shows that the DoFs in (1)-(2) are unisolvent for ${U_M}$.
		
		\begin{theorem}
			The DoFs in (1)-(2) are unisolvent for ${U_M}$.
			\label{unisolvence}
		\end{theorem}
		\begin{proof}
			There are twenty-seven DoFs {in (1)-(2). The continuity of functions in ${U_M}$} at $x_0$, $m_i$ and $d_i$($1\leqslant i\leqslant 3$) give six, twenty-five and two constraints, respectively. This implies
			\begin{equation*}
				\dim {U_M}\geqslant\sum_{i=1}^4\dim P_4\left(K_i\right)-\left(6+25+2\right)=27.
			\end{equation*}
			It suffices to show that if for some $u\in {U_M}$, the values of $u$, $\nabla u$, $\nabla^2u$ vanish at $x_i$ for $0\leqslant i\leqslant2$, and the values of $\partial_{\bm{n}}u$, $\partial_{\bm{tn}}^2u$, $\partial_{\bm{tt}}^2u$ vanish at $m_i$ for $1\leqslant i\leqslant3$, then $u$ must be $0$. The proof is divided into four steps.

			\textbf{Step 1 shows $\trace{\partial_{\bm{n}}u}{\partial M}=0$.} Take {the edge} $x_1m_1$ as an example. The fact that $\partial_{\bm{n}}u$, $\partial_{\bm{tn}}^2u$ vanish at $x_1$ and $m_1$ implies that $\partial_{\bm{n}}u$ {vanishes} on $x_1m_1$. Similar arguments show that $\partial_{\bm{n}}u$ vanishes on $\partial M$.
			
			\textbf{Step 2 shows $\trace{u}{\partial M}=0$.} Take {the edge} $x_1x_2$ as an example. Note that $u$, $\partial_{\bm{t}}u$, $\partial_{\bm{tt}}^2u$ vanish at $x_1$, $x_2$, and that $u$ is continuous at $m_1$. This shows $$u=\lambda_{m_1}^3q,\quad\text{on $x_1x_2$}.$$ with a piecewise linear continuous function $q$ on $x_1x_2$. Note that $q$ and $\partial_{\bm{t}}u$ are continuous at $m_1$ and that $\partial_{\bm{tt}}^2u(m_1)=0$. These facts imply that $q=0$ on $x_1x_2$. Therefore $u$  {vanishes} on $x_1x_2$. Similar arguments show that $u$ vanishes on $\partial M$.

%

			\textbf{Step 3 shows that $u=\lambda_{m_1}^2p$ on $M$ for some $p\in C(M)$.} Note that $u$, $\partial_{\bm{n}}u$ vanish on $\partial M$. This implies {\begin{equation}
					u=\lambda_{m_1}^2p
					\label{eq: u-p}
				\end{equation}
				with a piecewise quadratic function $p$ on $M$ that satisfies
				\begin{equation}
					\trace{p}{K_1}=c_1\lambda_{m_2}^2,\ \trace{p}{K_4}=c_2\lambda_{m_3}^2
					\label{eq: p-K1-K4}
				\end{equation}
				with some constants $c_1$, $c_2$}. The continuity of $u$ across the  {edges} $m_1m_2$, $m_1x_0$, $m_1m_3$ {and \eqref{eq: u-p} imply} that $p$ is continuous across these three edges. Therefore $p\in C(M)$,  {and $p\left(m_1\right)=0$.}
			
			\textbf{Step 4 shows $p=0$.} On each element $K_i$, {taking the first and second order derivatives of both sides of \eqref{eq: u-p} yields}
			\begin{equation}
			\nabla u=2\lambda_{m_1}p\nabla\lambda_{m_1}+\lambda_{m_1}^2\nabla p,
			\label{eq: nablau}
			\end{equation}
			\begin{equation}
			\nabla^2u=\lambda_{m_1}^2\nabla^2p+2\lambda_{m_1}\big(\nabla\lambda_{m_1}(\nabla p)^T+\nabla p(\nabla\lambda_{m_1})^T\big)+2p\big(\nabla\lambda_{m_1}(\nabla\lambda_{m_1})^T\big).
			\label{eq: nabla2u}
			\end{equation}
			{Here the gradient is considered as a column vector.} Since $\nabla u(m_1)=0$ and $p(m_1)=0$, \eqref{eq: nablau} shows $\nabla p(m_1)=0$. Besides, $\nabla^2u(x_0)=0$ and \eqref{eq: nabla2u} imply $p(x_0)=0$. These two facts together with $p(m_1)=0$ lead to $\trace{p}{x_0m_1}=0$. 
			
			The continuity of $\nabla^2 u$ at $m_1$, $p(m_1)=\nabla p(m_1)=0$ {and \eqref{eq: nabla2u}} imply that $\nabla^2p$ is continuous at $m_1$. Let $\bm{l}$ be the tangential vector of {the edge} $x_0m_1$, $\bm{t}$ and $\bm{n}$ be the tangential vector and the normal vector of $x_1x_2$. The fact $\trace{p}{x_0m_1}=0$ implies $\partial_{\bm{l}\bm{l}}^2p(m_1)=0$. Besides, the facts $\trace{p}{x_1x_2}=0$ and $\trace{\partial_{\bm{n}}p}{x_1x_2}=0$ imply {$$\partial_{\bm{t}\bm{t}}^2p(m_1)=\partial_{\bm{t}\bm{n}}^2p(m_1)=0.$$}Note that $\bm{l}\bm{l}^T,\bm{t}\bm{t}^T,\bm{n}\bm{t}^T+\bm{t}\bm{n}^T$ form a basis of $\Sym$. This shows $\nabla^2p(m_1)=0$. {Therefore $c_1=c_2=0$,} and $\trace{p}{K_1}=0$, $\trace{p}{K_4}=0$. 
			
			 {Note that $\nabla\lambda_{m_1}$ is continuous across the  {edges} $m_1m_2$ and {$m_1m_3$}. This,} the $C^1$ continuity of $u$ and \eqref{eq: nablau} imply $\nabla p$ is continuous across $m_1m_2$ and $m_1m_3$. This and $\trace{p}{K_1}=0$ imply {$\trace{p}{K_2}=c\lambda_{x_0}^2$, with $c\in\R$}. The fact $\trace{p}{x_0m_1}=0$ implies {$c=0$} and $\trace{p}{K_2}=0$. Similar arguments show $\trace{p}{K_3}=0$. Therefore $p=0$, $u=\lambda_{m_1}^2p=0$. This concludes the proof.

		\end{proof}
		
		
		{Lemma \ref{C1} and the proof of Theorem \ref{unisolvence} imply that} an $H^2$ conforming finite element space ${U_h}$ {can be} constructed in the following way:
		 {\begin{equation}
			\begin{aligned}
			{U_h}:=\big\{&u\in L^2(\Omega):\ \left.u\right|_{M}\in {U_M},\ \forall M\in\M_h,\ \text{$u$ is $C^2$ at {vertices of $\M_h$}},\\ 
			&\text{$\partial_{\bm{n}}u$, $\partial_{\bm{tn}}^2u$ and $\partial_{\bm{tt}}^2u$ are continuous at midpoints of {edges of $\M_h$}}\big\}.
			\end{aligned}
			\end{equation}}


		\subsection{The sequence}
		
		For a domain $\Omega\subset\R^2$, the elasticity sequence reads
		\begin{equation}
		P_1(\Omega)\subset C^{\infty}(\Omega)\stackrel{J}{\longrightarrow}C^{\infty}(\Omega;\Sym)\stackrel{\dv}{\longrightarrow}C^{\infty}(\Omega;\R^2)\longrightarrow0,
		\end{equation}
		and its Sobolev space version reads
		\begin{equation}
		P_1(\Omega)\subset H^2(\Omega)\stackrel{J}{\longrightarrow}H(\dv,\Omega;\Sym)\stackrel{\dv}{\longrightarrow}L^2(\Omega;\R^2)\longrightarrow0.
		\end{equation}
		
		Recall the finite element spaces $\Sigma_{2,h}$ {from Subsection~\ref{subsec: 2DP2-element}} and $V_{2,h}$ {from \eqref{eq: V-space}}. The following discrete elasticity sequence will {be proved} to be exact in this section:
		\begin{equation}
		P_1(\Omega)\subset {U_h}\stackrel{J}{\longrightarrow}\Sigma_{2,h}\stackrel{\dv}{\longrightarrow}V_{2,h}\longrightarrow0.
		\label{discrete_sequence}
		\end{equation}
		
		{The following lemma} characterizes the continuity {conditions} of the macro-element bubble functions, and Lemma \ref{Lem: inclusion} proves the inclusion relationship $J{U_h}\subseteq\Sigma_{2,h}$.
		
		\begin{lemma}
			Let the macro-element bubble function space {$\Sigma_{M,2,b}$} be given as in Section \ref{sec: 2Dmacrobubble}. It holds $\Sigma_{M,2,b}=\Sigma_{\partial M,2,0}^*$ with
			\begin{equation}
			\begin{aligned}
			\Sigma_{\partial M,2,0}^*=\big\{\sigma\in H_0\left(\dv,M;\Sym\right):\ \left.\sigma\right|_{K_i}\in P_2(K_i;\Sym),\ 1\leqslant i\leqslant 4,\\
			\text{$\sigma$ is continuous at $m_j$},\ 1\leqslant j\leqslant 3\big\}.
			\end{aligned}
			\end{equation}
			\label{Lem: macrobubble}
		\end{lemma}
		\begin{proof}
			The choice of macro-element bubble functions shows $\Sigma_{M,2,b}\subseteq\Sigma_{\partial M,2,0}^*$. It suffices to prove the converse.
			
			{Take any $\sigma\in\Sigma_{\partial M,2,0}^*$.} Note {that} $\sigma$ is continuous at $m_1$ and has vanishing normal component along {the edge} $x_1x_2$. This shows $\sigma(m_1)=a\bm{t}\bm{t}^T$, for some $a\in\R$. Here $\bm{t}$ is the tangential vector of $x_1x_2$. One can use the third kind of macro-element bubble functions in Section \ref{sec: 2Dmacrobubble} to eliminate this non-vanishing tangential-tangential component, i.e. $$\left(\sigma-a\Phi_{m_1}\bm{t}\bm{t}^T\right)(m_1)=0.$$Similar arguments also hold for $m_2$ and $m_3$. {This means {that} there exists $\sigma_1\in\Sigma_{M,2,b}$ such that 
				\begin{equation*}
				\left(\sigma-\sigma_1\right)(m_j)=0,\quad1\leqslant j\leqslant 3.
				\end{equation*}
				
				Note that $\left(\sigma-\sigma_1\right)\bm{n}$ on each interior edge $E_i$ is a quadratic function vanishing at two endpoints of this edge. Here $\bm{n}$ is the unit normal vector of $E_i$, with $1\leqslant i\leqslant3$. The second kind of macro-element bubble functions in Section \ref{sec: 2Dmacrobubble} can be used to eliminate this non-vanishing normal component, i.e. there exists $\sigma_2\in\Sigma_{M,2,b}$ such that $$\left(\sigma-\sigma_1-\sigma_2\right)\bm{n}=0,\quad\text{on $E_i$, with $1\leqslant i\leqslant 3.$}$$This implies that $\sigma-\sigma_1-\sigma_2$ is a piecewise bubble function on $M$ and $\sigma-\sigma_1-\sigma_2\in\Sigma_{M,2,b}$.  Therefore $\sigma\in\Sigma_{M,2,b}$.} This concludes the proof.

		\end{proof}

		%

		\begin{lemma}
			{It holds} $J{U_h}\subseteq\Sigma_{2,h}$.
			\label{Lem: inclusion}
		\end{lemma}
		\begin{proof}
			Given $u\in {U_h}$, let $\sigma=Ju$. Then $\trace{\sigma}{K}\in P_2\left(K;\Sym\right)$ for each $K\subset M$ and $M\in\M_h$. {Note that ${U_h}\subset H^2\left(\Omega\right)$ and $J$ maps $H^2\left(\Omega\right)$ to $H\left(\dv,\Omega;\Sym\right)$. These facts imply $\sigma\in H\left(\dv,\Omega;\Sym\right)$ and $\sigma \bm{n}$ is continuous across $\partial M$. Besides, the $C^2$ continuity of ${U_h}$ at macro-element vertices implies $\sigma$ is continuous at macro-element vertices.}


			{The continuity of $\sigma$ implies there exists $\sigma_1\in H^1\left(\Omega;\Sym\right)$ and $\trace{\tau}{K}\in P_2\left(K;\Sym\right)$ on each element $K$, such that $\sigma_1=\sigma$ at macro-element vertices and $\sigma_1\bm{n}=\sigma\bm{n}$ at three interior nodes of each macro-element edge $E$. This  {yields} that $\trace{\sigma_1\bm{n}}{E}=\trace{\sigma\bm{n}}{E}$.} Let $\sigma_2=\sigma-\sigma_1$, then on each macro-element $M$, $\trace{\sigma_2}{M}\in\Sigma_{\partial M,2,0}^*$. Lemma \ref{Lem: macrobubble} shows $\trace{\sigma_2}{M}\in\Sigma_{M,2,b}$. This shows $\sigma=\sigma_1+\sigma_2\in\Sigma_{2,h}$.
		\end{proof}
		
		\begin{theorem}
			The sequence \eqref{discrete_sequence} is exact.
		\end{theorem}
		\begin{proof}
			

			{Theorem~\ref{Thm: inf-sup-2D}} leads to $\dv\Sigma_{2,h}=V_{2,h}$. This implies {that} the discrete sequence \eqref{discrete_sequence} is exact if $\dim {U_h}-\dim\Sigma_{2,h}+\dim V_{2,h}\geqslant3$. Let $\#M$ be the number of the macro-elements. Let $\#V$, $\#E$ be the number of the vertices, edges of the macro-elements, respectively. Let $\#E^i$ be the number of the interior macro-element edges.


			The dimension of $\Sigma_{2,h}$ can be counted as follows:
			\begin{equation*}
			\begin{aligned}
			\dim\Sigma_{2,h}&\leqslant\left[3\left(\#V+\#E\right)+2\left(2\#E+3\#M\right)+3\cdot4\#M\right]+\#E^i\\
			&=3\#V+7\#E+\#E^i+18\#M.
			\end{aligned}
			\end{equation*}
			Note that $\dim {U_h}=6\#V+3\#E$ and $\dim V_{2,h}=24\#M$. The facts above show
			\begin{equation*}
			\dim {U_h}-\dim\Sigma_{2,h}+\dim V_{2,h}\geqslant3\#V-4\#E-\#E^i+6\#M.
			\end{equation*}
			The inequality above, combined with  {Euler's formula} $1+\#E=\#V+\#M$ and  {the} edge identity $3\#M=\#E+\#E^i$, show
			\begin{equation*}
			\dim {U_h}-\dim\Sigma_{2,h}+\dim V_{2,h}\geqslant3.
			\end{equation*}
			This concludes the proof.

		\end{proof}


		\section{Numerical tests}
		In the computation, the compliance tensor for the homogeneous isotropic  is given by 
		\begin{equation*}
		A\tau=\frac{1}{2\mu}\left(\tau-\frac{\lambda}{2\mu+n\lambda}{\rm tr}(\tau)\delta\right),
		\end{equation*}
		where $\delta$ denotes the identity matrix, and $\mu>0$, $\lambda>0$ are the  {Lam\'{e}} constants.
		\subsection{$k=2$ in 2D}\label{subsec: 2DP2numerics}
		We compute a $2D$ pure displacement problem on the unit square $\Omega=(0,1)^2$ with a homogeneous boundary condition that $u\equiv0$ on $\partial\Omega$. Let $\mu=1/2$ and $\lambda=1$ and the exact solution be
		\begin{align*}
		\begin{aligned}
		u=\begin{pmatrix}
		{\rm exp}(x-y)x(1-x)y(1-y)\\ \sin(\pi x)\sin(\pi y).
		\end{pmatrix}
		\end{aligned}
		\end{align*}
		
		In the computation, the level one mesh shown in Figure~\ref{fig:2Dmac} is obtained by a uniform bisection of the original mesh in Figure~\ref{fig:2Dini}. Each original mesh is refined into a half-sized  mesh uniformly, after that  a uniform  {bisection} leads to the macro-mesh. The convergence results are listed in Table~\ref{table:2Dresult}, which coincides with Theorem~\ref{Thm: 2Destimate}.

		\begin{figure}[h]
			\subcaptionbox{Original mesh.\label{fig:2Dini}}{\begin{tikzpicture}[line width=0.5pt,scale=1.3]
				\coordinate (A) at (0,0);
				\coordinate (B) at (0,2);
				\coordinate (C) at (2,2);
				\coordinate (D) at (2,0);
				
				\draw (A) -- (B) -- (C) -- cycle;
				\draw (A) -- (D) -- (C) -- cycle;

				\end{tikzpicture}
			}\quad\quad\quad\quad\quad\subcaptionbox{Macro-mesh.\label{fig:2Dmac}}{\begin{tikzpicture}[line width=0.5pt,scale=1.3]

				\coordinate (A) at (0,0);
				\coordinate (B) at (0,2);
				\coordinate (C) at (2,2);
				\coordinate (D) at (2,0);
				
				\draw (A) -- (B) -- (C) -- cycle;
				\draw (A) -- (D) -- (C) -- cycle;

				\coordinate (AB) at ($(A)!0.5!(B)$);
				\coordinate (BC) at ($(B)!0.5!(C)$);
				\coordinate (CD) at ($(C)!0.5!(D)$);
				\coordinate (DA) at ($(D)!0.5!(A)$);
				
				\draw (AB) -- (CD);
				\draw (BC) -- (DA);
				\draw (B) -- (D);

				\end{tikzpicture}}
			
			\caption{{Refinement.}}
			\label{fig:2Dexample}
		\end{figure}

		\begin{table}[ht]
			\begin{tabular}{c|cc|cc}
				\hline
				meshsize&$\|\sigma-\sigma_h\|_{L^2(\Omega)}$& rate & $\|u-u_h\|_{L^2(\Omega)}$& rate \\
				\hline
				1     &    0.45246    &         --- & 0.039392     &    ---\\ 
				0.5     &    0.17977     &   1.3316     & 0.019599   & 1.0071 \\
				0.25    &    0.025369  &      2.8251   & 0.0049733 &   1.9785\\ 
				0.125    &   0.0033584   &     2.9172   & 0.0012482  &  1.9943 \\
				0.0625   &   0.00042966   &     2.9665  & 0.00031237 &   1.9985 \\
				\hline
			\end{tabular}
			\caption{The errors for $k=2$ in 2D}
			\label{table:2Dresult}
		\end{table}
	
		\subsection{$k=3$ in 3D}\label{subsec: 3DP3numerics}
		We compute a $3D$ pure displacement problem on the unit cube $\Omega=(0,1)^3$ with a homogeneous boundary condition that $u\equiv0$ on $\partial\Omega$. Let $\mu=1/2$ and $\lambda=1$ and the exact solution be
		\begin{align*}
		\begin{aligned}
		u=\begin{pmatrix}
		2^4\\2^5\\2^6
		\end{pmatrix}x(1-x)y(1-y)z(1-z)
		\end{aligned}
		\end{align*}

		\begin{figure}[h]
			\subcaptionbox{Original mesh. \label{fig:3DP3ini}}{\begin{tikzpicture}[line width=0.5pt,scale=1.3]
				\coordinate (A) at (0,0);
				\coordinate (B) at (0,2);
				\coordinate (C) at (2,2);
				\coordinate (D) at (2,0);
				
				\coordinate (A1) at (0.9,0.3);
				\coordinate (B1) at (0.9,2.3);
				\coordinate (C1) at (2.9,2.3);
				\coordinate (D1) at (2.9,0.3);
				
				\draw (A) -- (B) -- (C) -- (D) -- cycle;
				\draw (C) -- (D) -- (D1) -- (C1) -- cycle;
				\draw (B) -- (B1) -- (C1) -- (C) -- cycle;
				
				\draw (D) -- (C1) -- (B);
				\draw (A) -- (C);
				\draw[dashed] (A) -- (A1) -- (B1);
				\draw[dashed] (A) -- (C1) -- (A1) -- (D1);
				\draw[dashed] (B1) -- (A) -- (D1);

				\end{tikzpicture}
			}\quad\quad\quad\quad\quad\subcaptionbox{Macro-mesh.\label{fig:3DP3mac}}{\begin{tikzpicture}[line width=0.5pt,scale=1.3]

				\coordinate (A) at (0,0);
				\coordinate (B) at (0,2);
				\coordinate (C) at (2,2);
				\coordinate (D) at (2,0);
				
				\coordinate (A1) at (0.9,0.3);
				\coordinate (B1) at (0.9,2.3);
				\coordinate (C1) at (2.9,2.3);
				\coordinate (D1) at (2.9,0.3);
				
				\draw (A) -- (B) -- (C) -- (D) -- cycle;
				\draw (C) -- (D) -- (D1) -- (C1) -- cycle;
				\draw (B) -- (B1) -- (C1) -- (C) -- cycle;
				
				\draw (D) -- (C1) -- (B);
				\draw (A) -- (C);
				\draw[dashed] (A) -- (A1) -- (B1);
				\draw[dashed] (A) -- (C1) -- (A1) -- (D1);
				\draw[dashed] (B1) -- (A) -- (D1);
				
				\coordinate (O) at ($(A)!0.5!(C1)$);
				
				\coordinate (Upper) at ($(B)!0.5!(C1)$);
				\coordinate (Lower) at ($(A)!0.5!(D1)$);
				\coordinate (Left) at ($(A)!0.5!(B1)$);
				\coordinate (Right) at ($(D)!0.5!(C1)$);
				\coordinate (Forward) at ($(A)!0.5!(C)$);
				\coordinate (Backward) at ($(A1)!0.5!(C1)$);

				\draw (B1) -- (C) -- (D1);
				\draw[dashed] (B) -- (A1) -- (D);
				\draw (B) -- (D);
				\draw[dashed] (B1) -- (D1);
				
				\draw[dashed] (B) -- (D1);
				\draw[dashed] (B1) -- (D);
				\draw[dashed] (A1) -- (C);
				
				\draw[dashed] (Upper) -- (Lower);
				\draw[dashed] (Left) -- (Right);
				\draw[dashed] (Forward) -- (Backward);
				
				\draw[color=black, fill=black]  (O) circle (1pt);
				
				\draw[color=black, fill=black]  (Upper) circle (1pt);
				\draw[color=black, fill=black]  (Lower) circle (1pt);
				\draw[color=black, fill=black]  (Left) circle (1pt);
				\draw[color=black, fill=black]  (Right) circle (1pt);
				\draw[color=black, fill=black]  (Forward) circle (1pt);
				\draw[color=black, fill=black]  (Backward) circle (1pt);

				\end{tikzpicture}}
			
			\caption{Refinement.}
			\label{fig:3DP3example}
		\end{figure}

		In the computation, the level one mesh shown in Figure~\ref{fig:3DP3mac} is obtained by dividing each tetrahedron of the original mesh  in Figure~\ref{fig:3DP3ini} into {four} tetrahedra. Each original mesh is refined uniformly, after that dividing as previously leads to the macro-mesh. The convergence results are listed in Table~\ref{table:3DP3result}, which coincides with Theorem~\ref{Thm: 3DP3estimate}.
		\begin{table}[ht]
			\begin{tabular}{c|cc|cc}
				\hline
				meshsize&$\|\sigma-\sigma_h\|_{L^2(\Omega)}$& rate & $\|u-u_h\|_{L^2(\Omega)}$& rate \\
				\hline
				1    &     0.082325     &        ---   &     0.028988      &   --- \\
				0.5   &     0.0051819    &    3.9898 &       0.0040798 &   2.8289 \\
				0.25   &    0.00034209   &      3.921  &     0.00051945 &   2.9735 \\
				
				\hline
			\end{tabular}
			\caption{The errors for $k=3$ in 3D}
			\label{table:3DP3result}
		\end{table}

		{\subsection{$k=2$ in 3D}\label{subsec: 3DP2numerics}
			We compute the same $3D$ problem {as} in Section \ref{subsec: 3DP3numerics}. In the computation, the level one mesh shown in Figure~\ref{fig:3DP2mac} is obtained by dividing each tetrahedron of the original mesh  in Figure~\ref{fig:3DP2ini} into twelve tetrahedra. Each original mesh is refined uniformly, after that dividing as previously leads to the macro-mesh. The convergence results are listed in Table~\ref{table:3DP2result}, {which asymptotically coincides with Theorem~\ref{Thm: 3DP2estimate}.}
			
			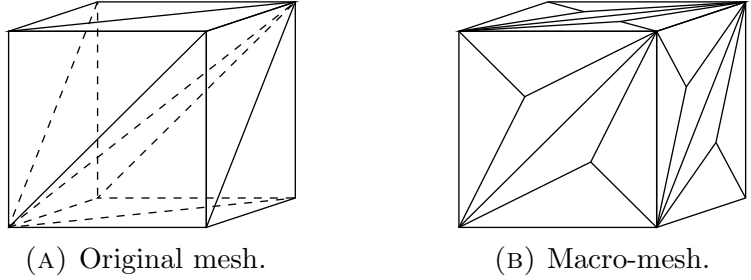
\begin{figure}[h]
				\subcaptionbox{Original mesh. \label{fig:3DP2ini}}{\begin{tikzpicture}[line width=0.5pt,scale=1.3]
					\coordinate (A) at (0,0);
					\coordinate (B) at (0,2);
					\coordinate (C) at (2,2);
					\coordinate (D) at (2,0);
					
					\coordinate (A1) at (0.9,0.3);
					\coordinate (B1) at (0.9,2.3);
					\coordinate (C1) at (2.9,2.3);
					\coordinate (D1) at (2.9,0.3);
					
					\draw (A) -- (B) -- (C) -- (D) -- cycle;
					\draw (C) -- (D) -- (D1) -- (C1) -- cycle;
					\draw (B) -- (B1) -- (C1) -- (C) -- cycle;
					
					\draw (D) -- (C1) -- (B);
					\draw (A) -- (C);
					\draw[dashed] (A) -- (A1) -- (B1);
					\draw[dashed] (A) -- (C1) -- (A1) -- (D1);
					\draw[dashed] (B1) -- (A) -- (D1);

					\end{tikzpicture}
				}\quad\quad\quad\quad\quad\subcaptionbox{Macro-mesh.\label{fig:3DP2mac}}{\begin{tikzpicture}[line width=0.5pt,scale=1.3]

					\coordinate (A) at (0,0);
					\coordinate (B) at (0,2);
					\coordinate (C) at (2,2);
					\coordinate (D) at (2,0);
					
					\coordinate (A1) at (0.9,0.3);
					\coordinate (B1) at (0.9,2.3);
					\coordinate (C1) at (2.9,2.3);
					\coordinate (D1) at (2.9,0.3);
					
					\draw (A) -- (B) -- (C) -- (D) -- cycle;
					\draw (C) -- (D) -- (D1) -- (C1) -- cycle;
					\draw (B) -- (B1) -- (C1) -- (C) -- cycle;
					
					\draw (D) -- (C1) -- (B);
					\draw (A) -- (C);

					\coordinate (m1) at (0.66666666667,1.3333333333333333);
					\draw (A) -- (m1) -- (B);
					\draw (m1) -- (C);
					\coordinate (m2) at (1.3333333333333333,0.66666666667);
					\draw (A) -- (m2) -- (D);
					\draw (m2) -- (C); 		
					\coordinate (m3) at (2.3,1.43333333333333333);
					\draw (D) -- (m3) -- (C1);
					\draw (m3) -- (C);		
					\coordinate (m4) at (2.6,0.8666666666666);
					\draw (D) -- (m4) -- (C1);
					\draw (m4) -- (D1);	 		
					\coordinate (m5) at (1.63333333333333333,2.1);
					\draw (C1) -- (m5) -- (B);
					\draw (m5) -- (C); 	 		
					\coordinate (m6) at (1.26666666666666667,2.2);
					\draw (C1) -- (m6) -- (B);
					\draw (m6) -- (B1);

					%
					%
					%
					%
					%
					%
					%
					%

					\end{tikzpicture}}
				
				\caption{Refinement.}
				\label{fig:3DP2example}
			\end{figure}

			\begin{table}[h]
				\begin{tabular}{c|cc|cc}
					\hline
					meshsize&$\|\sigma-\sigma_h\|_{L^2(\Omega)}$& rate & $\|u-u_h\|_{L^2(\Omega)}$& rate \\
					\hline
					1      &   0.90246        &       ---       &  0.13854      &     ---\\
					0.5  &       0.20592      &    2.1318     &   0.042811    &  1.6942 \\
					0.25    &     0.032968      &    2.6429    &    0.011251   &   1.9279 \\
					
					\hline
				\end{tabular}
				\caption{The errors for $k=2$ in 3D on the macro-mesh.}
				\label{table:3DP2result}
			\end{table}

			{For comparison,} we present the convergence results for the element of degree $k=2$ in Table~\ref{table:3DP2result-ini}. In this case the convergence order of the stress in $L^2$ norm is not optimal. This {justifies} the use of macro-elements to improve the convergence rate.
			
			\begin{table}[h]
				\begin{tabular}{c|cc|cc}
					\hline
					meshsize&$\|\sigma-\sigma_h\|_{L^2(\Omega)}$& rate & $\|u-u_h\|_{L^2(\Omega)}$& rate \\
					\hline
					1     &     1.5784       &         ---       &    0.25691     &      --- \\
					0.5     &      0.39872      &      1.985    &      0.085996 &     1.5789 \\
					0.25        &  0.083743        &   2.2513     &     0.023579   &   1.8668 \\
					0.125     &     0.018084  &         2.2113     &    0.0059988    &  1.9747 \\
					
					\hline
				\end{tabular}
				\caption{The errors for $k=2$ in 3D.}
				\label{table:3DP2result-ini}
			\end{table}

		}
		
		\bibliographystyle{siamplain}
		\bibliography{ref}

\begin{thebibliography}{10}

\bibitem{Adams}
{\sc S.~Adams and B.~Cockburn}, {\em A mixed finite element method for
  elasticity in three dimensions}, J. Sci. Comput., 25 (2005), pp.~515--521.

\bibitem{Arnold-Awanou}
{\sc D.~N. Arnold and G.~Awanou}, {\em Rectangular mixed finite elements for
  elasticity}, Math. Models Methods Appl. Sci., 15 (2005), pp.~1417--1429.

\bibitem{Arnold-Awanou-Winther}
{\sc D.~N. Arnold, G.~Awanou, and R.~Winther}, {\em Finite elements for
  symmetric tensors in three dimensions}, Math. Comp., 77 (2008),
  pp.~1229--1251.

\bibitem{Arnold-Falk-Winther}
{\sc D.~N. Arnold, R.~Falk, and R.~Winther}, {\em Mixed finite element methods
  for linear elasticity with weakly imposed symmetry}, Math. Comp., 76 (2007),
  pp.~1699--1723.

\bibitem{Arnold1984}
{\sc D.~N. Arnold, J.~D. Jr., and C.~P. Gupta}, {\em A family of higher order
  mixed finite element methods for plane elasticity}, Numer. Math., 45 (1984),
  pp.~1--22.

\bibitem{Arnold-Winther-conforming}
{\sc D.~N. Arnold and R.~Winther}, {\em Mixed finite element for elasticity},
  Numer. Math., 92 (2002), pp.~401--419.

\bibitem{Boffi-Brezzi-Fortin}
{\sc D.~Boffi, F.~Brezzi, and M.~Fortin}, {\em Reduced symmetry elements in
  linear elasticity}, Commun. Pure Appl. Anal., 8 (2009), pp.~95--121.

\bibitem{Boffi-Brezzi-Fortin2013}
{\sc D.~Boffi, F.~Brezzi, and M.~Fortin}, {\em Mixed finite element methods and
  applications}, Springer, Heidelberg, 2013.

\bibitem{ChenHuang3D2022}
{\sc L.~Chen and X.~Huang}, {\em A finite element elasticity complex in three
  dimensions}, Math. Comp., 91 (2022), pp.~2095--2127.

\bibitem{ChenHuang2022}
{\sc L.~Chen and X.~Huang}, {\em Finite elements for div- and divdiv-conforming
  symmetric tensors in arbitrary dimension}, SIAM J. Numer. Anal., 60 (2022),
  pp.~1932--1961.

\bibitem{Alfeld}
{\sc S.~H. Christiansen, J.~Gopalakrishnan, Guzm\'{a}n, and K.~Hu}, {\em A
  discrete elasticity complex on three-dimensional {A}lfeld splits}, Numer.
  Math.,  (2023).

\bibitem{Nodal}
{\sc S.~H. Christiansen, J.~Hu, and K.~Hu}, {\em Nodal finite element de {R}ham
  complexes}, Numer. Math., 139 (2018), pp.~411--446.

\bibitem{CGG}
{\sc B.~Cockburn, J.~Gopalakrishnan, and J.~Guzm\text{á}n}, {\em A new
  elasticity element made for enforcing weak stress symmetry}, Math. Comp., 79
  (2010), pp.~1331--1349.

\bibitem{Worsey}
{\sc S.~Gong, J.~Gopalakrishnan, J.~Guzm\'{a}n, and M.~Neilan}, {\em Discrete
  elasticity exact sequences on {W}orsey-{F}arin splits}, ESAIM: M2AN, 57
  (2023), pp.~3373--3402.

\bibitem{Gong}
{\sc S.~Gong, S.~Wu, and J.~Xu}, {\em New hybridized mixed methods for linear
  elasticity and optimal multilevel solvers}, Numer. Math., 141 (2019),
  pp.~569--604.

\bibitem{Hu2015trianghigh}
{\sc J.~Hu}, {\em Finite element approximations of symmetric tensors on
  simplicial grids in {$\mathbb{R}^n$}: {T}he higher order case}, J. Comput.
  Math., 33 (2015), pp.~283--296.

\bibitem{Rectangle}
{\sc J.~Hu}, {\em A new family of efficient conforming mixed finite elements on
  both rectangular and cuboid meshes for linear elasticity in the symmetric
  formulation}, SIAM J. Numer. Anal., 53 (2015), pp.~1438--1463.

\bibitem{HuLiangLin}
{\sc J.~Hu, Y.~Liang, and T.~Lin}, {\em Finite element grad grad complexes and
  elasticity complexes on cuboid meshes}, J. Sci. Comput., 50 (2024).

\bibitem{HuMa}
{\sc J.~Hu and R.~Ma}, {\em Partial relaxation of {$C^0$} vertex continuity of
  stresses of conforming mixed finite elements for the elasticity problem},
  Comput. Methods. Appl. Math., 21 (2021), pp.~89--108.

\bibitem{3DP3}
{\sc J.~Hu, R.~Ma, and Y.~Sun}, {\em A new mixed finite element for the linear
  elasticity problem in 3{D}}, J. Comput. Math.,  (2024).

\bibitem{HuZhang2014a}
{\sc J.~Hu and S.~Zhang}, {\em A family of conforming mixed finite elements for
  linear elasticity on triangular grids}, ar{X}iv, 1406.7457 (2014).

\bibitem{HuZhang2015tre}
{\sc J.~Hu and S.~Zhang}, {\em A family of symmetric mixed finite elements for
  linear elasticity on tetrahedral grids}, Sci. China Math., 58 (2015),
  pp.~297--307.

\bibitem{HuZhang2015trianglow}
{\sc J.~Hu and S.~Zhang}, {\em Finite element approximations of symmetric
  tensors on simplicial grids in {$\mathbb{R}^n$}: {T}he lower order case},
  Math. Models Methods Appl. Sci., 26 (2016), pp.~1649--1669.

\bibitem{Johnson-Mercier}
{\sc C.~Johnson and B.~Mercier}, {\em Some equilibrium finite element methods
  for two-dimensional elasticity problems}, Numer. Math., 30 (1978),
  pp.~103--116.

\bibitem{ScottZhang}
{\sc L.~R. Scott and S.~Zhang}, {\em Finite element interpolation of nonsmooth
  functions satisfying boundary conditions}, Math. Comp., 54 (1990),
  pp.~483--493.

\end{thebibliography}
		
	\end{document}